\documentclass[12pt]{article}
\usepackage{graphicx}
\usepackage{amsmath}
\usepackage{amssymb}
\usepackage{theorem}
\usepackage{dutchcal}

\usepackage{tikz}
\usetikzlibrary{arrows}
\usetikzlibrary{decorations.pathreplacing}

   \usepackage{hyperref}

\numberwithin{equation}{section}

\newtheorem{thm}{Theorem}[section]
\newtheorem{lemma}[thm]{Lemma}
\newtheorem{prop}[thm]{Proposition}
\newtheorem{cor}[thm]{Corollary}
{\theorembodyfont{\rmfamily}
\newtheorem{defn}[thm]{Definition}
\newtheorem{example}[thm]{Example}

\newtheorem{rmk}[thm]{Remark}
}

\newcommand{\qed}{\hfill \mbox{\raggedright \rule{.07in}{.1in}}}
 
\newenvironment{proof}{\vspace{1ex}\noindent{\bf
Proof}\hspace{0.5em}}{\hfill\qed\vspace{1ex}}

\newcommand{\R}{{\mathbb R}}

\newcommand{\N}{{\mathbb N}}
\newcommand{\E}{{\mathbb E}}
\newcommand{\PP}{{\mathbb P}}
\newcommand{\bbS}{{\mathbb S}}

\newcommand{\cI}{{\mathcal I}}
\newcommand{\cJ}{{\mathcal J}}
\newcommand{\cM}{{\mathcal M}}
\newcommand{\cN}{{\mathcal N}}
\newcommand{\cS}{{\mathcal S}}
\newcommand{\tcS}{{\widetilde{\mathcal S}}}

\newcommand{\tD}{{\widetilde D}}

\newcommand{\te}{{\tilde e}}
\newcommand{\tL}{{\tilde L}}
\newcommand{\tU}{{\widetilde U}}

\newcommand{\Leb}{\operatorname{Leb}}
\newcommand{\sgn}{\operatorname{sgn}}
\newcommand{\disc}{\operatorname{Disc}}
\newcommand{\diam}{\operatorname{diam}}
\newcommand{\range}{\operatorname{range}}

\begin{document}

\title{Convergence to decorated L\'evy processes in non-Skorohod topologies for dynamical systems}

\author{Ana Cristina Moreira Freitas
\thanks{Centro de Matem\'{a}tica \& Faculdade de Economia da Universidade do Porta, Rua Dr. Roberto Frias,  4200-464 Porto, Portugal} 
\and
Jorge Milhazes Freitas
\thanks{Centro de Matem\'{a}tica \& Faculdade de Ci\^encias da Universidade do Porto, Rua do Campo Alegre 687, 4169-007 Porto, Portugal}
\and
Ian Melbourne
\thanks{Mathematics Institute, University of Warwick, Coventry, CV4 7AL, UK}
\and
Mike Todd
\thanks{Mathematical Institute, University of St Andrews,
North Haugh, St Andrews, KY16 9SS, Scotland}
}


\maketitle

\begin{abstract}
We present a general framework for weak convergence to decorated L\'evy processes in enriched spaces of c\`adl\`ag functions for vector-valued processes arising in deterministic systems. Applications include uniformly expanding maps and unbounded observables as well as nonuniformly expanding/hyperbolic maps with bounded observables. The latter includes intermittent maps and dispersing billiards with flat cusps.
In many of these examples, convergence fails in all of the Skorohod topologies.
Moreover, the enriched space picks up details of excursions that are not recorded by Skorohod or Whitt topologies.
\end{abstract}

 \section{Introduction} 
 \label{sec-intro}

The classical central limit theorem (CLT) asserts convergence to a normal distribution with standard diffusion rate $n^{1/2}$. Donsker's weak invariance principle (WIP) gives weak convergence to the corresponding Brownian motion.
Brownian motion has continuous sample paths, so weak convergence can be taken in the space of continuous functions with the supremum norm.

There has been much interest across the physical sciences
(see for example~\cite{BL,GaspardWang88, Gouezel04, KGSSR, MZ15, Metzler, PVHS, SamorodnitskyTaqqu, Swinney, Whitt}) 
in ``anomalous diffusion'' and in particular in superdiffusive rates $n^{1/\alpha}$, $\alpha\in(0,2)$, with convergence to an $\alpha$-stable L\'evy process.
Such process have infinite variance and a dense set of discontinuities, exhibiting jumps of all sizes. It is customary to consider weak convergence in the space $D$ of c\`adl\`ag functions (right continuous with left limits).
Skorohod~\cite{Skorohod56} introduced various topologies for convergence in $D$,
and 
for a long time the Skorohod $\cJ_1$ topology was the topology of choice. 

Eventually, it became apparent that convergence in the $\cJ_1$ topology is too restrictive.
The first such examples were~\cite{AvramTaqqu92,BasrakKrizmanicSegers12} in the probability literature and~\cite{MZ15} in the dynamical systems literature, where convergence fails in $\cJ_1$ but holds in the weaker Skorohod $\cM_1$ topology on $D$.
Based on~\cite{AvramTaqqu92}, Whitt~\cite[p.~xii]{Whitt} and 
Jakubowski~\cite{Jakubowski07} 
respectively wrote:
{\advance\leftmargini -1.3em
\begin{quote}
Thus, while the $\cJ_1$ topology sometimes cannot be used, the
$\cM_1$ topology can almost always be used. Moreover, the extra strength of the $\cJ_1$
topology is rarely exploited. Thus, we would be so bold as to suggest that, \emph{if only
one topology on the function space D is to be considered, then it should be the $\cM_1$
topology}.
\end{quote}
}
{\advance\leftmargini -1.3em
\begin{quote}
All these reasons bring interest also to the
weaker Skorokhod's topologies $\cJ_2$, $\cM_1$ and $\cM_2$. Among them practically only
the topology $\cM_1$ proved to be useful.
\end{quote}
}

On the other hand, Whitt~\cite{Whitt} anticipated the need to move beyond the Skorohod topologies, and furthermore to replace $D$ by an enriched space of ``decorated'' c\`adl\`ag functions. The enriched spaces in~\cite{Whitt} were denoted by $E$ and $F$. These spaces permit weak convergence in situation where convergence fails in any Skorohod topology. Moreover, 
they keep track of various details which are lost in the usual Skorohod topologies.

The picture changed further as a result of the papers~\cite{BasrakKrizmanic14,MV20}
which gave first examples where the $\cM_2$ topology is the appropriate one.
Moreover, the examples considered in~\cite{MV20}, namely dispersing billiards with flat cusps, demonstrate emphatically that none of the Skorohod topologies are adequate in general. 
Examples from one-dimensional dynamics where convergence again fails in all Skorohod topologies are given in~\cite{Freitas2Toddsub}.

In this paper, we prove a general result on weak convergence to $\alpha$-stable L\'evy processes in a decorated 
c\`adl\`ag space. Our framework is general enough to incorporate all known examples arising in uniformly and nonuniformly hyperbolic dynamics. In particular, we cover intermittent maps and billiards with flat cusps (bounded observables) and uniformly expanding maps (unbounded observables). Our results are formulated using a space $F'$ introduced in~\cite{Freitas2Toddsub} which improves upon the spaces in~\cite{Whitt} and achieves three 
goals: 
\begin{itemize}
\item[(i)] We obtain weak convergence results when none are possible using the Skorohod topologies. For billiards with one flat cusp as considered in~\cite{MV20}, we obtain convergence in $F'$ to a decorated L\'evy process for typical H\"older observables, whereas such a result is false for the Skorohod topologies~\cite{JMPVZ21}. 
\item[(ii)] We keep track of fine behaviour concerning excursions during jumps, as partially illustrated in Figure~\ref{fig:M1}. This includes behaviour that is not detected by the spaces in~\cite{Whitt}.
\item[(iii)] Restricting to $d=1$ for convenience, the Skorohod topologies have the property that important functionals such as $\psi(u)(t)=\sup_{s\in[0,t]}u(s)$ are continuous on $D$ and hence preserve weak convergence. Many of our examples in Section~\ref{sec:informal} have ``overshoots'' (as illustrated in Figures~\ref{fig:ex5} and~\ref{fig:cusps}) meaning that weak convergence cannot be preserved by such functionals. Consequently, the processes in such examples cannot converge in a Skorohod topology (nor in any topology on $D$ for which such functionals are continuous).  However, such functionals $\psi$ are continuous on the enriched space $F'$ (see Remark~\ref{rmk:psi}) and so we obtain a large class of functionals that preserve weak convergence of enriched processes.
\end{itemize}

The approach in this paper, building on~\cite{Freitas2Toddsub}, is to consider decorated L\'evy process obtained by attaching suitable profiles $P:[0,1]\to\R^d$ that keep precise track of the excursion during each jump.

\begin{rmk}
For intermittent map examples, considered in~\cite{MZ15} for scalar observables and~\cite{CFKM20} for vector-valued observables, our general theory applies when the observable $v$ is nonvanishing at at least one of the most neutral fixed points. 

For dispersing billiards with flat cusps, we require moreover that  the excursions at all of the flattest cusps are in distinct directions. The general case is the topic of work in progress.
\end{rmk}

The remainder of this paper is organised as follows.
Section~\ref{sec:informal} provides an informal and nontechnical
description of numerous examples covered by our theory. This serves as an illustration of the differences between the various Skorohod topologies on $D$, as well as the improvements
arising from the theory in this paper.
In Section~\ref{sec:rv}, we recall background material on regular variation in $\R^d$.
In Section~\ref{sec:F'}, we define the topological space $F'$ of decorated c\`adl\`ag functions~\cite{Freitas2Toddsub}.
Our main result, Theorem~\ref{thm:main}, on weak convergence in $F'$ is stated and proved in Section~\ref{sec:main}.
In Section~\ref{sec:ex}, we revisit various examples covered by Theorem~\ref{thm:main}.

\paragraph{Notation}
We use ``big O'' and $\ll$ notation interchangeably, writing $a_n=O(b_n)$ or $a_n\ll b_n$
if there are constants $C>0$, $n_0\ge1$ such that
$a_n\le Cb_n$ for all $n\ge n_0$.
As usual, $a_n=o(b_n)$ means that $a_n/b_n\to0$ and
and $a_n\sim b_n$ means that $a_n/b_n\to1$.

For c\`adl\`ag functions $u:[0,1]\to\R^d$ and $\tau\in (0,1]$, we let
$u(\tau^-)=\lim_{\epsilon\downarrow 0} u(t-\epsilon)$ denote the left hand limit of $u$ at $\tau$.

We denote by $p_k:\R^d\to\R$ the projection onto the $k$-th coordinate for $k=1, \ldots, d$.
Let $x,y\in\R$.
Throughout this paper, $x\wedge y=\min\{x, y\}$, $x\vee y=\max\{x, y\}$
and $[x,y]$ is the line segment
$[x\wedge y,x\vee y]$. 
For $x,y\in\R^d$, we define the product segment
$[[x,y]]=[p_1x,p_1y]\times\dots\times[p_dx,p_dy]$.

\section{Informal description and illustrative examples}
\label{sec:informal}

Throughout this section,
$T:M\to M$ is a measure-preserving dynamical system defined on a probability space $(M,\mu)$ and $v:M\to\R^d$ is a measurable observable.
For specified $\alpha\in(0,2)$, we define the sequence of c\`adl\`ag processes $W_n\in D([0,1],\R^d)$,
\[
W_n(t)=n^{-1/\alpha}\sum_{j=0}^{[nt]}v\circ T^j,
\]
on $M$.

\subsection{One dimensional maps}

One-dimensional maps $T:M\to M$, $M=[0,1]$, already provide a wide variety of different weak convergence properties. 
We begin with three examples, each of which exhibits convergence in the Skorohod $\cM_1$ topology, but for which much further information is gained by considering convergence in the enriched space $F'$.

\begin{example} \label{ex:1}
Tyran-Kami\'nska~\cite{TyranKaminska10} initiated the study of weak convergence to $\alpha$-stable L\'evy processes in deterministic dynamical systems, focusing on the standard Skorohod $\cJ_1$ topology on $D$. A specific example studied in~\cite{TyranKaminska10} is the Gauss map 
$T_1x=1/x \bmod1$
which arises in the study of continued fractions.
For the scalar observable $v(x)=[1/x]$, it was shown taking $\alpha=1$
that 
$W_n$ converges weakly (after appropriate centring) in the $\cJ_1$ topology to a totally-skewed (one-sided) $1$-stable L\'evy process.
\end{example}

\begin{example} \label{ex:2}
Our second example is a class of Pomeau-Manneville intermittent maps~\cite{PomeauManneville80} studied in~\cite{LSV99}
\[
T_2x=\begin{cases} x(1+2^{1/\alpha}x^{1/\alpha}) & 0\le x<\frac12 \\
2x-1 & \frac12 \le x\le 1
\end{cases} \;.
\]
These maps possess a neutral fixed point at $x=0$ (with $T_2'(0)=1$) that becomes stickier as $\alpha$ decreases resulting in anomalous behaviour.
For $\alpha>1$, there is a unique absolutely continuous invariant probability measure $\mu$.
Let $v$ be a scalar H\"older observable with $\int v\,d\mu=0$ and $v(0)\neq0$.
For $\alpha\ge2$, we have central limit theorem behaviour with normalisation $n^{1/2}$ for $\alpha>2$ and normalisation $(n\log n)^{1/2}$ for $\alpha=2$.
In the case of interest here, $\alpha\in(1,2)$, Gou\"ezel~\cite{Gouezel04} proved that $W_n(1)$ converges in distribution to a totally-skewed $\alpha$-stable law.
Convergence to the corresponding $\alpha$-stable L\'evy process was proved in the Skorohod $\cM_1$ topology in~\cite{MZ15}. It had already been noted in~\cite{TyranKaminska10} that convergence was impossible in the $\cJ_1$ topology.
\end{example}

In the first example, large jumps for $W_n(t)$ arise at separated times $t=j/n$ whenever $T_1^jx$ is near zero. In the second example, increments of $W_n$ are small (bounded by $n^{-1/\alpha}|v|_\infty$). However, when $T_2^jx$ is near the neutral fixed point at $0$, there are several successive values of $j$ for which the increments are all close to $n^{-1/\alpha}v(0)$ and these accumulate into a large jump. In $\cJ_1$, large jumps in the limit have to be approximated by large jumps of almost the same size at almost the same instant of time, so the $\cJ_1$ topology is appropriate for the first example but not the second. 

\begin{example} \label{ex:3}
Our third example, considered in~\cite{Gouezel_doub}, is provided by the doubling map $T_3=2x\bmod1$ with scalar observable
$v(x)=x^{-1/\alpha}$ where $\alpha\in(0,1)$. 
Again, $W_n(1)$ converges in distribution to a totally-skewed $\alpha$-stable law and  $W_n$ converges to the corresponding L\'evy process in the $\cM_1$ topology but not in $\cJ_1$.
However, the situation is quite different from that for $T_2$ where the increments for $W_n$ during an excursion near the neutral fixed point at $0$ limit on a vertical line segment. For $T_3$, the increments near the hyperbolic fixed point at $0$ limit on a sequence of a large jumps, decreasing geometrically in size at rate $2^{-1/\alpha}$ 
(see Example~\ref{ex:3+} for details).
\end{example}

For the three examples above,
the limiting excursions are (1) a pair of points, (2) a line segment, (3) a geometric sequence of points. These are illustrated in Figure~\ref{fig:M1}.
Note that all three examples converge in $\cM_1$. The first example is distinguished by converging also in $\cJ_1$. However, the second and third examples cannot be distinguished at the level of Skorohod topologies,  nor by the spaces in~\cite{Whitt} as discussed in~\cite[Section 2.3]{Freitas2Toddsub}. This is the issue addressed by the space $F'$ in~\cite{Freitas2Toddsub}. 

We obtain decorated processes in $F'$ by associating
to each jump, a \emph{profile}, namely a c\`adl\`ag function $P:[0,1]\to\R$ with $P(0)=0$ and $|P(1)|=1$, as shown in Figure~\ref{fig:M1}. 
Our main result, Theorem~\ref{thm:main}, gives convergence in $F'$ of $W_n$ to the decorated L\'evy process.

\begin{figure}[h]
\begin{tikzpicture}[thick, scale=0.5]


\draw[ -] (1,0) -- (9,0);
\draw[-] (1, 0)--(1, -0.3);
\draw (1, -0.3) node[below] {\small $0$};
\draw[-] (9, 0)--(9, -0.3);
\draw (9, -0.3) node[below] {\small $1$};

\draw[ -o ] (1, 1) -- (5,1);

\fill (5,7) circle (5.5pt);
\draw[ -] (5, 7) -- (9,7);

\draw (3, 4) node[above] {\small $P_{I(\tau)}$};

\draw[ - >] (1,10) -- (9,10);
\draw[-] (5, 10)--(5, 9.7);
\draw (5, 9.7) node[below] {\small $\tau$};

\draw[ dashed] (1, 11) -- (3.5,11);
\draw[-](3.5, 11)--(5, 11);
\draw[-](5, 17)--(6.5, 17);
\draw[ dashed] (6.5, 17) -- (9,17);


\draw[ -] (11,0) -- (19,0);
\draw[-] (11, 0)--(11, -0.3);
\draw (11, -0.3) node[below] {\small $0$};
\draw[-] (19, 0)--(19, -0.3);
\draw (19, -0.3) node[below] {\small $1$};

\draw (13, 4) node[above] {\small $P_{I(\tau)}$};

\draw[ - ] (11, 1) -- (19,7);

\draw[ - >] (11,10) -- (19,10);
\draw[-] (15, 10)--(15, 9.7);
\draw (15, 9.7) node[below] {\small $\tau$};

\draw[ dashed] (11, 11) -- (13.5,11);
\draw[-](13.5, 11)--(15, 11);
\draw[-](15, 11)--(15, 17);
\draw[-](15, 17)--(16.5, 17);
\draw[ dashed] (16.5, 17) -- (19,17);


\draw[ -] (21,0) -- (29,0);
\draw[-] (21, 0)--(21, -0.3);
\draw (21, -0.3) node[below] {\small $0$};
\draw[-] (29, 0)--(29, -0.3);
\draw (29, -0.3) node[below] {\small $1$};

\draw (26, 3) node[above] {\small $P_{I(\tau)}$};

\draw[ -o ] (21, 1) -- (22.5,1);

\draw[ -o ] (22.5, 4) -- (24.3,4);
\fill (22.5,4) circle (5.5pt);

\draw[ -o ] (24.3, 5.5) -- (25.8,5.5);
\fill (24.3,5.5) circle (5.5pt);

\draw[ -o ]  (25.8,6.2)--(27, 6.2);
\fill (25.8,6.2) circle (5.5pt);

\draw[ -o ]  (27, 6.6)--(27.8, 6.6);
\fill (27,6.6) circle (5.5pt);

\draw[ -o ]  (27.8, 6.8)--(28.3, 6.8);
\fill (27.8,6.8) circle (5.5pt);

\draw[ - ]  (28.3, 7)--(29, 7);

\fill (5,7) circle (5.5pt);
\draw[ -] (5, 7) -- (9,7);

\draw[ - >] (21,10) -- (29,10);
\draw[-] (25, 10)--(25, 9.7);
\draw (25, 9.7) node[below] {\small $\tau$};

\draw[ dashed] (21, 11) -- (23.5,11);
\draw[-](23.5, 11)--(25, 11);

\fill (25,14) circle (5.5pt);
\fill (25,15.5) circle (5.5pt);
\fill (25,16.2) circle (5.5pt);
\fill (25,16.6) circle (5.5pt);
\fill (25,16.8) circle (5.5pt);

\draw[-](25, 17)--(26.5, 17);
\draw[ dashed] (26.5, 17) -- (29,17);

\end{tikzpicture}

\caption{Limiting excursions (first row) at some time $\tau$ and profiles (second row) for the three one-dimensional examples $T_1$ (Gauss map), $T_2$ (intermittent map with $v(0)>0$), $T_3$ (doubling map). Each excursion/profile corresponds to one jump of the limiting L\'evy process. Each excursion is a subset of a vertical line and is the image of the corresponding profile $P= P_{I(\tau)}:[0,1]\to\R$ suitably scaled ($I(\tau)$ picks the correct profile at $\tau$ and is defined in Section~\ref{sec:main}).}
\label{fig:M1}

\end{figure}
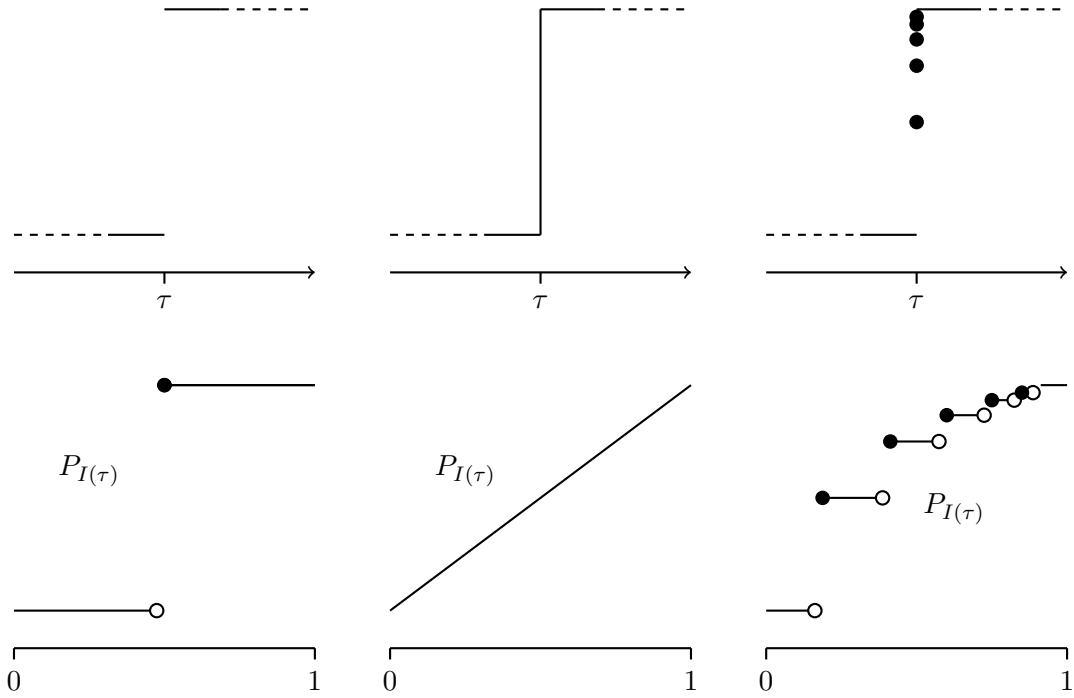

The limiting L\'evy process in these three examples are totally-skewed with jumps that all have the same sign (positive for $T_1$ and $T_3$ and $\sgn v(0)$ for $T_2$).  
The next example includes two-sided L\'evy processes as well as the vector-valued case.

\begin{example} \label{ex:4}
Continuing Example~\ref{ex:2}, let
$T_4:M\to M$, $M=[0,1]$ be an intermittent map with finitely many neutral fixed points $x_1,\dots,x_k\in M$ where 
$T_4x\approx x+c_j (x-x_j)^{1+1/\alpha_j}$ for $x\approx x_j$ ($c_j>0$),
where $\alpha_1=\min\alpha_j\in(1,2)$.
For an explicit example, see~\cite[eq.~(1.5) and Lemma~6.3]{CFKM20}.
Let $v:M\to\R$ be H\"older with $v(x_1)\neq0$.
Then we obtain 
convergence in $\cM_1$ to an $\alpha_1$-stable L\'evy process.
If $\alpha_1=\alpha_2$ and $v(x_1)v(x_2)\neq0$, then the L\'evy process is two-sided.

Now define profiles $P_\pm(t)\equiv \pm t$. We attach $P_+$ to each positive jump and $P_-$ to each negative jump, scaled by the size of the jump.
Our results guarantee convergence in $F'$ to the L\'evy process enriched in this manner.
Note that it is not required that $v(x_j)\neq0$ for all $j$; it suffices that $v(x_j)\neq0$ for at least one neutral fixed point $x_j$ with $\alpha_j$ least.

We can also consider vector-valued observables $v:M\to\R^d$ 
with $d\ge2$ as in~\cite{CFKM20}.
Let $i\in\cI$ be the set of indices $i\in\{1,\dots,k\}$ such that
$\alpha_i=\alpha_1$ least and $v(x_i)\neq0$. 
We assume that $\cI\neq\emptyset$.
By~\cite{CFKM20}, we obtain convergence in $\cM_1$ to an $\alpha_1$-stable L\'evy process with jumps in the directions
$\omega_i=v(x_i)/|v(x_i)|$, $i\in\cI$. We attach the profile $P_i(t)\equiv t\omega_i$ to jumps in directions $\omega_i$, scaled by the size of the jump.
As shown in Example~\ref{ex:4+}, the results in this paper yield convergence in $F'$ to this enriched L\'evy process.
\end{example}

The $\cM_1$ topology suffices for the examples mentioned so far. Moreover, the profiles can be recovered from the excursion combined with the knowledge that convergence holds in $\cM_1$. In our next example, convergence fails in all Skorohod topologies and the profile contains information that cannot be gleaned from the excursion.

\begin{example} \label{ex:5}
Consider the map $T_5:M\to M$ studied in~\cite[Example~2.7]{Freitas2Toddsub}, 
\[
T_5x=3x\bmod1, \qquad v(x)=|x-\tfrac18|^{-2}-|x-\tfrac38|^{-2}.
\]
Large values of $v$ with alternating sign arise when $T_5^jx$ is close to the repelling period two orbit $\{\frac18,\frac38\}$.

In this example, $W_n(1)$ converges to a symmetric $\frac12$-stable law, and the normalised excursions consist of the points
$\{1-(-\frac19)^j;\,j=0,1,2,\ldots\}\cup\{1\}$ at positive jumps and
$\{-1+(-\frac19)^j;\,j=0,1,2,\ldots\}\cup\{-1\}$ at negative jumps.
Note that the excursions span $[0,\pm\frac{10}{9}]$, thereby overshooting the span $[0,\pm1]$ of the jumps. 
The excursion and profile for positive jumps are shown in Figure~\ref{fig:ex5}.
The profile contains considerable extra information, indicating that the size of the steps during one jump decrease in size with oscillating sign; the limiting excursion records the size of the steps but not the order in which they occur.
See Example~\ref{ex:5+} for further details.

Let $L_\alpha$ denote the corresponding $\frac12$-stable L\'evy process.
The functional 
\[
\psi:D\to D, \qquad
\psi(u)(t)=\sup_{s\in[0,t]}u(s),
\]
 is continuous in the Skorohod topologies.
In examples like the current one, with overshooting excursions, it is clear that the limit of $\psi(W_n)$ is unrelated to $\psi(L_\alpha)$ since $L_\alpha$ does not see the overshoots. Hence it follows from 
 the continuous mapping theorem that $W_n$ does not converge weakly to $L_\alpha$ in any Skorohod topology on $D$.
However, the enriched process records the overshoots and we recover
continuity of such functionals from $F'$ to $D$.
\end{example}

\begin{figure}[h]
\centering
\begin{tikzpicture}[thick, scale=0.7]


\draw[ - >] (1,0) -- (9,0);
\draw[-] (5, 0)--(5, -0.3);
\draw (5, -0.3) node[below] {\small $\tau$};

\draw[ dashed] (1, 0.6) -- (3.5,0.6);
\draw[-](3.5, 0.6)--(5, 0.6);
\draw[-](5, 5.6)--(6.5, 5.6);
\draw[ dashed] (6.5, 5.6) -- (9,5.6);

\fill (5,7) circle (4pt);
\fill (5,4.6) circle (4pt);
\fill (5,6.1) circle (4pt);
\fill (5,5.3) circle (4pt);
\fill (5,5.8) circle (4pt);


\draw[ -] (11,0) -- (19,0);
\draw[-] (11, 0)--(11, -0.3);
\draw (11, -0.3) node[below] {\small $0$};
\draw[-] (19, 0)--(19, -0.3);
\draw (19, -0.3) node[below] {\small $1$};

\draw (16, 2) node[above] {\small $P_{I(\tau)}=P_1$};

\draw[ -o ] (11, 0.6) -- (12.5,0.6);

\draw[ -o ] (12.5, 7) -- (14.3,7);
\fill (12.5,7) circle (4pt);

\draw[ -o ] (14.3, 4.6) -- (15.8,4.6);
\fill (14.3,4.6) circle (4pt);

\draw[ -o ]  (15.8,6.1)--(17, 6.1);
\fill (15.8,6.1) circle (4pt);

\draw[ -o ]  (17, 5.3)--(17.8, 5.3);
\fill (17,5.3) circle (4pt);

\draw[ -o ]  (17.8, 5.8)--(18.3, 5.8);
\fill (17.8,5.8) circle (4pt);

\draw[ - ]  (18.3, 5.6)--(19, 5.6);

\end{tikzpicture}

\caption{Limiting excursion (left) and profile (right) at a positive jump for Example~\ref{ex:5}. The profile $P_1:[0,1]\to\R$ with $P_1(0)=0$, $P_1(1)=1$ corresponds to jumps initiated at $x=\frac18$ (see Example~\ref{ex:5+}); the other possibility being negative jumps initiated at  $x=\frac38$ (with profile $P_{-1}=-P_1$).
The second horizontal line in the profile is at height $\frac{10}{9}$, overshooting the range $[0, 1]$ of the profile.}
\label{fig:ex5}
\end{figure}
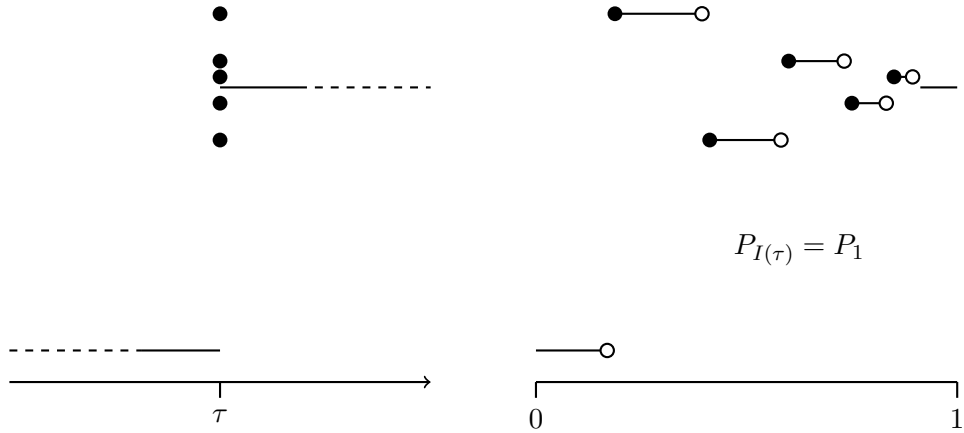

\begin{rmk}
In situations where the results in this paper apply, we obtain necessary and sufficient conditions for convergence in the $\cM_1$ and $\cM_2$ topologies by arguments
in~\cite{MV20}
for $\cM_1$ and in~\cite{JMPVZ21} for $\cM_2$.
Convergence holds in $\cM_2$ if and only if 
$P(t)$ is contained in the line segment $\ell_P$  joining $0$ to $P(1)$ for all $t\in(0,1)$ and each profile $P$. 
Convergence holds in $\cM_1$ if and only if 
in addition $t\mapsto P(t),\ t\in[0,1]$ is monotone in $\ell_P$ for each $P$.

If there exists a profile $P$ and a $t\in[0,1]$ such that $P(t)\not\in \ell_P$ then convergence fails in all the Skorohod topologies. 
This occurs naturally (and typically for $d\ge2$) in the billiard examples in Subsection~\ref{sec:cusps}, as well as in Example~\ref{ex:5}.
\end{rmk}

\subsection{Billiards with flat cusps}
\label{sec:cusps}

Planar dispersing billiards~\cite{ChernovMarkarian} were introduced by Sinai~\cite{Sinai70} and are based on deterministic Lorentz gas models~\cite{Lorentz05}. They 
are known to satisfy numerous classical statistical limit laws 
such as the CLT  and WIP~\cite{BunimSinai81,BunimovichSinaiChernov91}. Billiards with cusps were treated by~\cite{BalintChernovDolgopyat11} who obtained convergence to a normal distribution/Brownian motion but with anomalous diffusive rate $(n\log n)^{1/2}$ instead of the usual $n^{1/2}$ normalisation.

Jung \& Zhang~\cite{JungZhang18} proved convergence to an $\alpha$-stable law for planar dispersing billiards with a flat cusp.
The billiard table $Q\subset\R^2$ has a boundary consisting of at least three $C^3$ curves
with a cusp formed by two of these curves $\Gamma_\pm$. In
coordinates $(s, z) \in\R^2$, the cusp lies at $(0, 0)$ and $\Gamma_\pm$ are tangent to the $s$-axis
at $(0, 0)$. Moreover, 
$\Gamma_\pm = \{(s, \pm\beta^{-1} s^\beta )\}$
close to $(0, 0)$, 
where $\beta > 2$. 

The phase space of the billiard map (collision map) $T$ is given by $M= \partial  Q \times[0, \pi ]$,
with coordinates $(r, \theta )$ where $r$ denotes arc length along $\partial  Q$ and $\theta$ is the angle between
the tangent line of the boundary and the collision vector in the clockwise direction. There
is a natural ergodic invariant probability measure $d\mu = (2|\partial Q|)^{-1} \sin \theta \,dr \,d\theta$ on $M$,
where $|\partial Q|$ is the length of $\partial Q$.

Let $v:M\to\R$ be a H\"older observable with $\int_M v\,d\mu=0$.
By~\cite{JungZhang18}, $W_n(1)$ converges weakly to a totally-skewed $\alpha$-stable law with $\alpha=\beta/(\beta-1)\in(1,2)$.
The case of multiple cusps was considered in~\cite{JungPeneZhang20}.
We refer to~\cite{JungPeneZhang20} for precise details of the configuration space; in particular it is assumed that no trajectory runs directly between  the vertices of two cusps.

Convergence to the corresponding L\'evy process is considered in~\cite{MV20,JungPeneZhang20,JMPVZ21}. 
Again, $\cJ_1$ convergence is impossible since the jumps are bounded.
If $v$ has constant sign on each cusp, then convergence holds in the $\cM_1$ topology. 
However, a much wider range of convergence properties is possible due to the fact that the cusp (which is a single point $(0,0)$ in configuration space) is a union of two line segments
\[
\{(r_+,\theta):0\le\theta\le\pi\}
\cup\{(r_-,\theta):0\le\theta\le\pi\}
\]
in phase space. Here, $r_\pm\in\Gamma_\pm$ denotes the arc length coordinates of $(0,0)$.

At each of the flattest cusp (those with largest $\beta$), we associate a continuous profile 
\[
P:[0,1]\to\R^d, \qquad P(t)=\frac12\int_0^t\{v(r_+,\theta)+v(r_-,\pi-\theta)\}(\sin\theta)^{1/\alpha}\,d\theta,
\]
as depicted in Figure~\ref{fig:cusps}.
We require that $P(1)\neq0$ for each $P$ and normalise so that $|P(1)|=1$.
In general, $P$ may have overshoots for $d=1$, see Figure~\ref{fig:cusps}(c), and overshoots are typical for $d\ge2$. Hence the billiard example provides many instances where convergence fails in all Skorohod topologies. 
In Section~\ref{sec:Young}, we apply our results to show (currently under the assumption that $P(1)$ is distinct for distinct flattest cusps) that convergence holds in $F'$ to an enriched L\'evy process.

\begin{figure}[htb]
\centering
\includegraphics[scale=.7]{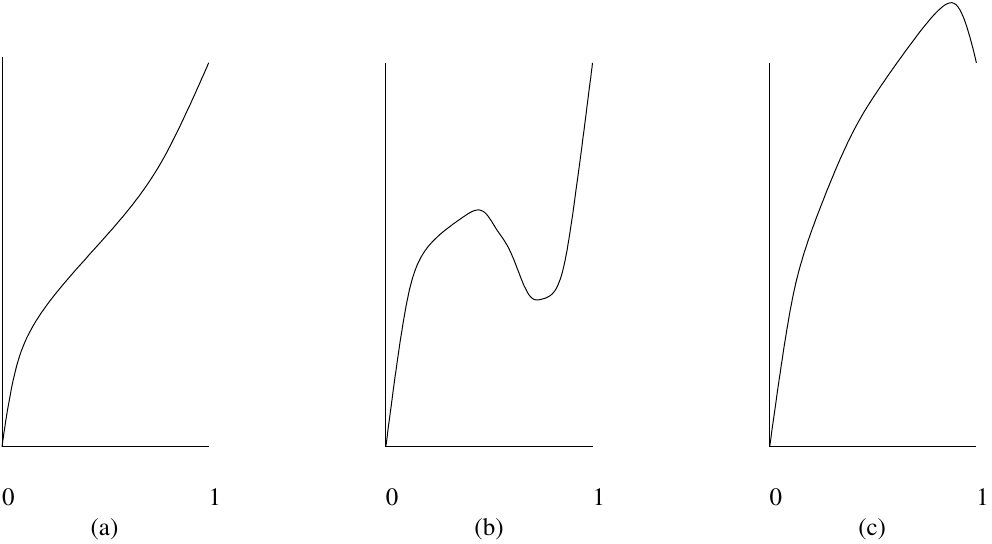}
\caption{Different possible shapes of the profile at a cusp for
billiards with flat cusps for a scalar observable $v:M\to\R$.
(a)  Convergence holds in the $\cM_1$ topology;
(b) Convergence holds in the $\cM_2$ topology but not in the $\cM_1$ topology;
(c) Convergence fails in all Skorohod topologies but holds in the enriched space $F'$.}
\label{fig:cusps}
\end{figure}

\section{Regular variation in $\R^d$ and spectral measures}
\label{sec:rv}

In this section, we recall some basic material on regularly varying vector-valued functions and the notion of spectral measure~\cite[Section~2.3]{SamorodnitskyTaqqu}.

Let $\bbS^{d-1} = \{x \in \R^d : |x| = 1 \}$ denote the unit sphere in $\R^d$.
(Throughout, $|\cdot|$ denotes the Euclidean norm.)

\begin{defn} \label{defn:rv}
An $\R^d$-valued random variable $Z$ is
    \emph{regularly varying} with order $\alpha\in(0,2)$ if
    there exists a Borel probability measure $\nu$
    on
    $\bbS^{d-1}$, called the \emph{spectral measure}, such that
\[
        \lim_{t \to \infty}
        \frac{\PP(|Z| > \lambda t, \ Z / |Z| \in E)}{\PP(|Z| > t)}
        = \lambda^{-\alpha} \nu(E)
\]
    for all $\lambda > 0$ and all Borel sets $E \subset \bbS^{d-1}$ with $\nu(\partial E) = 0$.
\end{defn}

Taking $E=\bbS^{d-1}$, we have that $|Z|$ is a scalar regularly varying
function.
Hence there exists a slowly varying function $\mathcal{l}:[0,\infty)\to(0,\infty)$
such that 
\[
\PP(|Z|>t)=t^{-\alpha}\mathcal{l}(t).
\]

Suppose that $Z$ is regularly varying as in Definition~\ref{defn:rv} and that
either $\alpha\in(0,1)$ or $\alpha\in(1,2)$ and
$\E Z=0$. 
Let $Z_1,Z_2,\ldots$ be a sequence of i.i.d.\ random variables distributed as $Z$. Choose $b_n\sim (n\mathcal{l}(b_n))^{1/\alpha}$. Then
\[
b_n^{-1}\sum_{j=0}^{n-1}Z_j   \to_w G_\alpha,
\]
where 
$G_\alpha$ is a $d$-dimensional $\alpha$-stable law with 
characteristic function
\[
\E\, e^{is\cdot G_\alpha}= \exp\Big\{-\int_{\bbS^{d-1}}|s\cdot x|^\alpha\Big(1-i\sgn(s\cdot x)
\tan\frac{\pi\alpha}{2}\Big)\cos\frac{\pi\alpha}{2}\,\Gamma(1-\alpha)\,d\nu(x)\Big\}
\]
for $s\in\R^d$. The random variable $Z$ is said to be in the domain of attraction of $G_\alpha$.

Now let $\tL_\alpha\in D([0,1],\R^d)$ denote the $d$-dimensional $\alpha$-stable L\'evy process corresponding to the stable law $G_\alpha$.
Also, define the process $W_n^Z$ by
\[
W^Z_n(t)=b_n^{-1}\sum_{j=0}^{[nt]-1}Z_j .
\]
Then $W^Z_n\to_w \tL_\alpha$ in the Skorohod $\cJ_1$ topology.

\begin{rmk} The strong $\cJ_1$ topology on $D([0,1],\R^d)$ is metrised 
by 
\[
d_{\cJ_1}(u_1,u_2)=\inf_\lambda \Big(\sup_{t\in[0,1]}|u_1(\lambda(t))-u_2(t)|+\sup_{t\in[0,1]}|\lambda(t)-t|\Big),
\]
where the infimum is over the set of continuous strictly increasing bijections $\lambda:[0,1]\to [0,1]$.
There is also a weak $\cJ_1$ topology defined by working coordinatewise (which allows $d$ different parametrisations $\lambda$). The weak and strong topologies coincide for $d=1$ and are different for $d\ge2$.
Throughout this paper, by $\cJ_1$ we mean strong $\cJ_1$.

\end{rmk}

\section{Decorated c\`adl\`ag space $F'$}
\label{sec:F'}

In this section, we recall the definition of the topological space 
$F'=F'([0,1],\R^d)$ introduced in~\cite{Freitas2Toddsub}. We make liberal use of the notation introduced at the end of the Introduction.

Let $D=D([0, 1], \R^d)$ be the space of c\`adl\`ag functions defined on 
$[0,1]$.  For $u\in D$, we denote by $\disc_u\subset(0,1)$ the set of 
discontinuities of $u$. 

The decorated c\`adl\`ag space $F'$ is defined to be the space of excursion triples $\big(u, \cS, \{e^\tau\}_{\tau\in \cS}\big)$
where 
\begin{itemize}
\item $u\in D$, 
\item $\cS$ is an at most countable subset of $(0,1)$ containing $\disc_u$,
\item $e^\tau\in D$ satisfies
$e^\tau(0)=u(\tau^-)$ and $e^\tau(1)=u(\tau)$ for each
$\tau\in \cS$. 
\item 
For all $\epsilon >0$, there exist only finitely many $\tau\in \cS$ such that $\diam\range e^\tau>\epsilon$.
\end{itemize}
(The second and fourth conditions are automatic for $u\in D$ if $\cS=\disc_u$
and $\range e^\tau\subset [[u(\tau^-),u(\tau)]]$.)

The remainder of this section is devoted to defining the appropriate topology on~$F'$.
To do this, it is useful to consider two further spaces $E$ and $\tD$.

The space $E=E([0, 1], \R^d)$ introduced by Whitt~\cite[Sections~15.4 and 15.5]{Whitt} is the space of triples
$
\big(u,\cS,\{K^\tau\}_{\tau\in \cS}\big)
$
where 
\begin{itemize}
\item $u\in D$, 
\item $\cS$ is an at most countable subset of $(0,1)$ containing $\disc_u$, 
\item 
$K^\tau$ is a compact connected subset of $\R^d$ containing at least $u(\tau^-)$ and $u(\tau)$
for each $\tau\in \cS$,  
\item 
 For all $\epsilon >0$, there exist only finitely many $\tau\in \cS$ such that $\diam K^\tau>\epsilon$. 
\end{itemize}

We may identify each element $\big(u,\cS,\{K^\tau\}_{\tau\in\cS}\big)\in E$ with the set-valued function
\[
\hat u(t)=\begin{cases}
    K^t& \text{if } t\in \cS\\
    \{u(t)\}              & t\in[0,1]\setminus\cS
\end{cases}, 
\]
and its graph $\Gamma_{\hat u}=\{(t,z)\in [0,1]\times\R^d : z\in \hat u(t)\}$.  

For elements of $E$, the associated graph $\Gamma_{\hat u}$ is a compact set. 
Recall that for compact sets $A,B\subset \R^n$, the Hausdorff distance between $A$ and $B$ is given as
$$
H(A,B)=
\sup_{x\in A}\inf_{y\in B}|x-y| \vee \sup_{y\in B}\inf_{x\in A}|x-y|.
$$
We endow $E$ with the Hausdorff metric by setting
\[
d_E(\hat u_1, \hat u_2)=H(\Gamma_{\hat u_1},\Gamma_{\hat u_2}),
\quad \hat u_1,\hat u_2 \in E.
\]

Next, we introduce $\tD= \tD([0, 1], \R^d)= D/{\sim}$ where $u_1\sim u_2$ if there exists a reparametrisation $\lambda:[0, 1]\to [0, 1]$, i.e.\ a continuous strictly increasing bijection, such that $u_1\circ \lambda=u_2$.  Denote the equivalence class of $u$ by $[u]$.  We define
\[
d_{\tD}([u_1], [u_2]) = \inf_{\lambda}\sup_{t\in[0,1]}|u_1(\lambda(t))- u_2(t)|,
 \quad u_1,u_2 \in D,
\]
where the infimum is over the set of continuous strictly increasing bijections $\lambda:[0,1]\to [0,1]$ (this could be thought of as the induced metric from the $J_1$ metric on $\tD$).
We abuse notation within $\tD$ by writing $u$ to refer to both a representative of its equivalence class $[u]$ and the equivalence class itself. 
(Note that two elements of $\tD$ can be close even if $\lambda$ is far from the identity, so $\tD$ is quite different from $D$ with the $\cJ_1$ topology.)

We define projections 
\[
\pi_E:F'\to E, \qquad \pi_\tD:F'\to \tD
\]
as follows.

The projection $\pi_E$ is given by
\[
\pi_E\big(u, \cS, \{e^\tau\}_{\tau\in \cS})=\big(u, \cS, \{K^\tau\}_{\tau\in \cS}\big)
\]
 where
$$
K^\tau=\Big[\inf_{t\in [0,1]} p_1 e^{\tau}(t), \sup_{t\in [0,1]} p_1e^{\tau}(t)\Big]\times\cdots\times \Big[\inf_{t\in [0,1]} p_de^{\tau}(t), \sup_{t\in [0,1]} p_de^{\tau}(t)\Big].
$$

To define $\pi_\tD\big(u, \cS, \{e^\tau\}_{\tau\in \cS}\big)$, 
write $\cS=\{\tau_m:m\in\kappa\}$ where $\kappa\subset\{1,2,\dots\}$ is an at most countable (possibly empty) indexing set.
Define $s=\sum_{m\in\kappa}m^{-2}$.
Insert an interval $I_m$ of length $m^{-2}$ after each $\tau_m$ to obtain an interval of length $1+s$. Define 
$\tilde u:[0,1+s]\to\R^d$ to coincide with $u$ on $[0,1+s]\setminus\bigcup_m I_m$\footnote{In formulas, let $\varphi:[0,1]\to [0,1+s]\setminus\bigcup_m I_m$ be the piecewise smooth bijection with $\varphi'=1$ on $[0,1]\setminus\cS$.
Then $\tilde u(t)=u(\varphi^{-1}(t))$ on $[0,1+s]\setminus\bigcup_m I_m$.}  and 
to coincide with the appropriate time-scaled version of $e^{\tau_m}$ on 
$I_m$. (So if $I_m=[a,a+m^{-2}]$, then $\tilde u(a+t)=e^{\tau_m}(m^2t)$ for $0\le t\le m^{-2}$.)
Define $\pi_\tD\big(u, \cS, \{e^\tau\}_{\tau\in \cS}\big)(t) = \tilde u(t(1+s))$.

	We can now define a pseudometric on $F'$ by setting
	\[
	d_{F'}(\check u_1,\check u_2)=d_E(\pi_E \check u_1, \pi_E \check u_2)+ 
	d_\tD(\pi_\tD \check u_1, \pi_\tD \check u_2),
\quad \check u_1,\check u_2\in F'. 
	\]

	\begin{rmk} The space $F'$ with the topology defined here is separable (but not complete),
	see~\cite[Proposition~A.3(a)]{Freitas2Toddsub}.
	\end{rmk}

\begin{rmk} \label{rmk:psi}
For $d=1$, consider the maximum process functional $\psi:F'\to D$, 
given by
\[
\psi(\check u)(t)=\sup_{s\in[0,t]}\max\{(\pi_E\check u)(s)\},
\quad \check u\in F',\,t\in[0,1].
\]
This is continuous, showing that we recover a suitable class of continuous functionals preserving weak convergence in $F'$.
See~\cite[eq.~(5.5) and Theorem~15.5.1]{Whitt} for the corresponding situation in the Whitt space $E$.
\end{rmk}

	\section{Main theorem}
	\label{sec:main}

	In this section, we state and prove the main theoretical result of the paper.

	Let $T:M\to M$ be an ergodic measure-preserving transformation on a probability space $(M,\mu)$ and let $X\subset M$ be a measurable subset with $\mu(X)>0$.
	Define the first return time
	\[
	R:X\to{\mathbb Z}^+, \qquad R(x)=\inf\{n\ge1:T^nx\in X\}
	\]
	and the first return map
	\[
	f=T^R:X\to X, \qquad fx=T^{R(x)}x.
	\]
	We assume throughout that $R\in L^1$.
	The normalised restriction $\mu_X$ of $\mu$ restricted to $X$ is
	an ergodic $f$-invariant probability measure on $X$.

	Fix finitely many unit vectors 
	$\omega_i\in\bbS^{d-1}$, $i\in\cI$, where $\cI$ is a finite indexing set.
	For notational convenience, suppose that $1\in\cI$.
	Also, we fix a finitely supported spectral measure 
	\[
	\nu=\sum_{i\in\cI}a_i\delta_{\omega_i}
	\]
	on $\bbS^{d-1}$, where
	$a_i>0$ and $\sum_{i\in\cI}a_i=1$.
Choose $b_n\sim(n\mathcal{l}(b_n))^{1/\alpha}$ as in Section~\ref{sec:rv}, and
	let $\tL_\alpha$ denote the $\alpha$-stable L\'evy process with spectral measure $\nu$.

	Let $v:M\to\R^d$ be a vector-valued observable, and set 
	$v_k=\sum_{j=0}^{k-1}v\circ T^j$. We define the induced observable
	\[
	V=v_R:X\to\R^d, \qquad V=\sum_{j=0}^{R-1}v\circ T^j.
	\]
	Also, define the processes $W_n$, $W^V_n\in D([0,1],\R^d)$ on $M$ and $X$ respectively,
	\begin{equation} \label{eq:Wn}
	W_n(t)=b_n^{-1}\sum_{j=0}^{[nt]-1}v\circ T^j, \qquad
	W^V_n(t)=b_n^{-1}\sum_{j=0}^{[nt]-1}V\circ f^j.
	\end{equation}
	Our main hypothesis is that $W_n^V\to_{\mu_X}\tL_\alpha$ in $D$ with the $\cJ_1$ topology.\footnote{We write $\to_{\mu_X}$ to denote weak convergence, emphasising the probability space on which $W_n^V$ are defined.}

	Define the $\alpha$-stable L\'evy process $L_\alpha=\big(\int_X R\,d\mu_X\big)^{-1/\alpha} \tL_\alpha$. Our aim is to prove weak convergence in $F'$ of $W_n$ to $L_\alpha$.  To make sense of this, we first need to embed $W_n$ and $L_\alpha$ as decorated processes $W_n^{F'}$ and $L_\alpha^{F'}$ in $F'$.

	\paragraph{Embedding of $W_n$ in $F'$.}
	We embed $W_n$ in $F'$ in a somewhat arbitrary (harmless) manner by attaching trivial excursions.
	Given $u\in D$, let $\Delta u(\tau)=u(\tau)-u(\tau^-)$.
	We define elements $W_n^{F'}\in F'$ on $(M,\mu)$,
	\[
	W_n^{F'}=\big(W_n,\cS_n,\{e_n^\tau\}_{\tau\in\cS_n}\big)
	\]
	where $\cS_n=\{\frac{j}{n},\,1\le j\le n-1\}$ 
	and $e_n^\tau:[0,1]\to\R^d$ is given by
	\[
	e_n^\tau(t)=W_n(\tfrac{j-1}{n})+1_{[\frac12,1]}(t)\Delta W_n(\tfrac{j}{n})
	\quad\text{for $\tau=\tfrac{j}{n}\in\cS_n$.}
	\]

	\paragraph{Embedding of $L_\alpha$ in $F'$.}

	Let $P_i$, $i\in\cI$, be a finite collection of 
	\emph{profiles}, namely c\`adl\`ag functions $P_i\in D([0,1],\R^d)$ 
	with $P_i(0)=0$ and $P_i(1)=\omega_i$.
	We now describe how to embed $L_\alpha$ in $F'$
	by adjoining the profiles $P_i$, suitably scaled, at each discontinuity of $L_\alpha$.

	Let $\cS=\disc_{L_\alpha}$.
	For each $\tau\in\cS$, we can express $\Delta L_\alpha(\tau)$ uniquely in the form
	$\Delta L_\alpha(\tau)=|\Delta L_\alpha(\tau)| \omega_{I(\tau)}$ with $I(\tau)\in\cI$.
	Define
	\[
	L_\alpha^{F'}=\big(L_\alpha,\cS,\{e^\tau\}_{\tau\in\cS}\big)
	\]
	where the excursion
	$e^\tau:[0,1]\to\R^d$, $\tau\in\cS$, is given by
	\[
	e^\tau(t)=L_\alpha(\tau^-)+|\Delta L_\alpha(\tau)|P_{I(\tau)}(t).
	\]

	\paragraph{Hypotheses.}

	As already mentioned, our main hypothesis is that $W_n^V\to_{\mu_X}\tL_\alpha$ in $D$ with the $\cJ_1$ topology.
	We require one more assumption linking the dynamics to the profiles.
Define $\Pi:\R^d\to D$ by setting
$\Pi(y)=P_i$ when $y/|y|$ is closest to $\omega_i$.
If $y/|y|$ is equidistant from two distinct $\omega_i$, or $y=0$, set $\Pi(y)=0$.
	We define functions $\xi,\,\zeta\in D([0,1],\R^d)$ on $X$,
	\[
	\xi(t)=v_{[tR]},\qquad
	\zeta(t)= |V|\Pi(V)(t)+t\big\{V-|V|\Pi(V)(1)\big\}.
	\]
	Our second main hypothesis is that
	\begin{equation} \label{eq:main}
	b_n^{-1}\max_{0\le j\le n} d_\tD(\xi,\zeta)\circ f^j \to_{\mu_X} 0.
	\end{equation}

	\begin{thm} \label{thm:main}
	Assume
	that $W_n^V\to_{\mu_X} \tL_\alpha$ in $D$ with the Skorohod $\cJ_1$ topology, and suppose that hypothesis~\eqref{eq:main} is satisfied.
	Then
	$W_n^{F'}\to_\mu L_\alpha^{F'}$ in $F'$.
	\end{thm}

	In the remainder of this section, we prove Theorem~\ref{thm:main}.

	\begin{rmk} \label{rmk:zeta}
A more natural choice when $\cI=\{1\}$ is to take $\zeta=VP_1$.
Our definition of $\zeta$ has the advantage that it treats all cases simultaneously.
For example,
        in the case $d=1$, $\cI=\{\pm1\}$, $\omega_\pm=\pm1$, we obtain 
        $\zeta=|V|P_{\sgn V}$.

        For $d=1$ with $\cI=\{1\}$ and
        $\omega_1=1$, we have
        $\zeta(t)=\begin{cases} VP_1(t) & V>0 \\ -VP_1(t)+2tV & V\le 0 \end{cases}$.
        The strange looking definition of $\zeta$ for $V\le0$ is unimportant since in practice $V$ will be large and positive in such situations.
	\end{rmk}

\begin{rmk}
We can adapt our proof to work when $R$ is a generalised inducing time, rather than necessarily a first return.  
This can be done by going to the corresponding Young tower and considering the first return to the base as in \cite[Section 4]{MV20}.
\end{rmk}

	\subsection{Initial elements of the proof}
	\label{sec:outline}

	The first step is to use
	ideas of strong distributional convergence~\cite{Zweimuller07} to reduce from 
	weak convergence w.r.t.\ $\mu$ to
	weak convergence w.r.t.\ $\mu_X$.

	\begin{lemma} \label{lem:SDC}
	$d_{F'}(W_n^{F'}\circ T,W_n^{F'})\to_\mu 0$.
	\end{lemma}

	\begin{proof}
	For $t\in(\frac{j}{n},\frac{j+1}{n})$, we have $W_n(t)=b_n^{-1}v_j$.
	Hence on this interval,
	\[
	W_n\circ T(t)=b_n^{-1}v_j\circ T=b_n^{-1}v_{j+1}-b_n^{-1}v
	=W_n(t+\tfrac1n)-b_n^{-1}v.
	\]
	This means that the values of 
	$W_n\circ T|_{(0,\frac{n-1}{n})}$ match up with those of 
	$W_n|_{(\frac{1}{n},1)}$ within error $b_n^{-1}|v|$ after a horizontal displacement of $\frac1n$.
	Hence the contribution to $d_E$ is at most $\frac1n+b_n^{-1}|v|$
	and the contribution to $d_\tD$ is at most $b_n^{-1}|v|$.
	We also have the estimates
	\[
	\sup_{[\frac{n-1}{n},1]}|W_n\circ T(t)-W_n(t)|\le b_n^{-1}(|v|\circ T^n+|v|).
	\]
	Hence
	\[
	d_{F'}(W_n^{F'}\circ T,W_n^{F'})\le \tfrac1n+2b_n^{-1}|v|
	+2b_n^{-1}|v|\circ T^n.
	\]
	The result follows since $|v|\circ T^n=_\mu |v|$ and $b_n^{-1}|v|\to0$ a.e.
	\end{proof}

	\begin{cor} \label{cor:SDC}
	To prove that $W_n^{F'}\to_\mu L_\alpha^{F'}$ in $F'$, it suffices
	to prove that $W_n^{F'}\to_{\mu_X} L_\alpha^{F'}$ in $F'$.
	\end{cor}

	\begin{proof}
	We have verified~\cite[Condition~(1)]{Zweimuller07} in Lemma~\ref{lem:SDC}.
	Hence the result follows from~\cite[Theorem~1]{Zweimuller07}.
	\end{proof}

	For $k\ge0$, define the \emph{lap number}
	\[
	N_k: X\to\N, \qquad N_k =\sum_{\ell=1}^k 1_X\circ T^\ell=\max\{n\ge0\colon R_n\le k\}\le k.
	\]
	where $R_n=\sum_{j=0}^{n-1}R\circ f^j$.

	\begin{prop} \label{prop:N}
	$\lim_{n\to\infty} n^{-1}\max_{1\le k\le n}N_k=\big(\int_X R\,d\mu_X\big)^{-1}$ a.e.\ on $(X,\mu_X)$.
	\end{prop}

	\begin{proof} 
	By definition of the lap number, $R_{N_n}\le n\le R_{N_{n+1}}$ so
	$n/N_n\to \int_X R\,d\mu_X$ a.e.\ by the pointwise ergodic theorem.
	Hence $n^{-1}N_n\to \big(\int_X R\,d\mu_X\big)^{-1}$ a.e.\  and the result follows easily.
	\end{proof}

	As in~\cite{MZ15,MV20,CFKM20}, we define the sequence of processes
	$U_n\in D$ on the probability space $(X,\mu_X)$,
	\[
	U_n(t)=b_n^{-1}\sum_{k=0}^{N_{[nt]}-1}V\circ f^k.
	\]
	These are rescaled versions of $W_n^V$ with 
	jumps occurring at 
	$$t_{n,j}=R_j/n$$ where $j=N_{[tn]}$. 

	\begin{lemma} \label{lem:U}
	$U_n\to_{\mu_X} L_\alpha$ in $D$ with the $\cJ_1$ topology.
	\end{lemma}

	\begin{proof}
	This is a consequence of the fact that
	$W_n^V\to_{\mu_X} \tL_\alpha$ in $D$ with the $\cJ_1$ topology.
	For completeness, we give the main steps in the argument following~\cite[Lemma~3.4]{MZ15}.
	Throughout $D$ is endowed with the $\cJ_1$ topology (rather than the $\cM_1$ topology used in~\cite{MZ15}).

	For $n \ge 1$ and $t\in[0,1]$, we let $\kappa_n(t) = n^{-1}N_{[tn]}$.
	Then $U_n (t) = W^V_n(\kappa_n (t))$ on $X$.
	We regard $U_n$, $W^V_n$, $L_\alpha$, $\tL_\alpha$ and $\kappa_n$ as random elements of $D$. Note that $\kappa_n\in D_\uparrow = \{g \in D:
	g(0) \ge 0 \text{ and $g$ nondecreasing}\}$. Let $\kappa$ denote the constant random element of $D$ given by $\kappa(t)(x) = t/\int_X R\,d\mu_X$.
	By Proposition~\ref{prop:N}, $\kappa_n(\cdot)(x) \to \kappa(\cdot)(x)$ uniformly on $[0,1]$ for $\mu_X$-a.e.\ $x\in X$.
	Hence, $\kappa_n \to_{\mu_X} \kappa$ in $D$.
	But then we automatically get
	$(W^V_n , \kappa_n ) \to_{\mu_X} (\tL_\alpha, \kappa)$ in $D^2$
	since 
	$W^V_n\to_{\mu_X} \tL_\alpha$ in $D$ and
	the limit $\kappa$ of the second component is deterministic.

	The composition map $D \times D_\uparrow \to D$, $(g, v)\mapsto g \circ v$, is continuous at every pair $(g, v)$ with
	$v \in C_{\Uparrow} = \{g \in D: g(0) \ge 0 \text{ and $g$ strictly increasing and continuous}\}$. By the continuous mapping theorem
	$U_n=W^V_n \circ \kappa_n \to_{\mu_X} \tL_\alpha \circ \kappa=L_\alpha$ in $D$ as required.~
	\end{proof}

	Define the functional
	\[
	\chi:D\to F', \qquad \chi u=(u,\disc_u,\{e_u^\tau\}_{\tau\in\disc_u}),
	\]
	where $e_u^\tau:[0,1]\to\R^d$, $\tau\in\disc_u$, is given by
	\[
	e_u^\tau(t)=u(\tau^-)+|\Delta u(\tau)|\Pi(\Delta u(\tau))(t)
+t\big\{\Delta u(\tau)-|\Delta u(\tau)|\Pi(\Delta u(\tau))(1)\big\}.
	\]
	In particular, $L_\alpha^{F'}=\chi L_\alpha$.
	Moreover, $L_\alpha$ is a continuity point of $\chi$ with probability one.
	(The discontinuity points of $\chi$ arise when $\Delta u(\tau)/|\Delta u(\tau)|$ is equidistant to distinct $\omega_i$.)

\begin{cor} \label{cor:U}
	$\chi U_n \to_{\mu_X} L_\alpha^{F'}$ in $F'$.
\end{cor}

\begin{proof}
This is immediate from
	Lemma~\ref{lem:U} by the continuous mapping theorem.
\end{proof}

\paragraph{Strategy for the remainder of the proof}
	The interval $[0,1]$ splits into subintervals
	$[t_{n,j},t_{n,j+1}]$, $0\le j\le N_n$,
	Notice that
	$1\in [t_{n,N_n},t_{n,N_n+1})$.
	For simplicity, we may suppose that $1\in (t_{n,N_n},t_{n,N_n+1})$ since this countable set of events occurs with probability one. 
	It is convenient to consider
	the final interval with $j=N_n$ separately.
	In Subsection~\ref{sec:inc}, we 
	show that
	 $d_{F',[t_{n,N_n},1]}\big(W_n^{F'},\chi U_n \big)  \to_{\mu_X}0$.
	Then in Subsection~\ref{sec:comp}, we 
	show that
	 $d_{F',[0,t_{n,N_n}]}\big(W_n^{F'},\chi U_n \big)  \to_{\mu_X}0$.
Combined, we obtain that
	 $d_{F'}\big(W_n^{F'},\chi U_n \big)  \to_{\mu_X}0$.
	By separability of $F'$, it then follows from
Corollary~\ref{cor:U} that $W_n^{F'}\to_{\mu_X} L_\alpha^{F'}$.
	By Corollary~\ref{cor:SDC},
	$W_n^{F'}\to_{\mu} L_\alpha^{F'}$ 
	completing the proof of Theorem~\ref{thm:main}.

	\paragraph{Convention}
	Recall that
	\(
	d_{F'}(\check u_1,\check u_2)=d_E(\pi_E\check u_1,\pi_E\check u_2)+
	d_{\tD}(\pi_{\tD}\check u_1,\pi_{\tD}\check u_2)
	\)
for $\check u_1,\check u_2\in F'$.
	When we write
	 $d_{E,J}\big(\pi_E\check u_1,\pi_E\check u_2\big)$,
	this means that we compute the graphs $\pi_E \check u_1, \pi_E\check u_2$ on $[0,1]$, restrict the graphs to $J\subset[0,1]$, and then compute the Hausdorff distance. Similarly for
	 $d_{\tD,J}\big(\pi_\tD \check u_1,\pi_\tD \check u_2\big)$
	 and $d_{F',J}\big(\check u_1,\check u_2\big)$.

\vspace{2ex}
We will require the following standard consequence of the pointwise ergodic theorem.
\begin{prop} \label{prop:ET}
Suppose that $H\in L^p(X)$, $p\ge1$.
Then $n^{-1/p}\max_{0\le j\le n}H\circ f^j=0$ a.e.\ on $(X,\mu_X)$.
\qed
\end{prop}

	\subsection{Incomplete excursion on $[t_{n,N_n},1]$}
	\label{sec:inc}

	\begin{prop} \label{prop:inc}
	$d_{F',[t_{n,N_n},1]}(W_n^{F'},\chi U_n )\to_{\mu_X}0$.
	\end{prop}

	\begin{proof}
	Write $t_{n,N_n}=\frac{j_n^*}{n}$ where $j_n^*\in\{0,\dots,n-1\}$. 
	Let $k\in\{1,\dots,d\}$.
	Restricted to $[\frac{j_n^*}{n},1]$, 
	the graph $\pi_E p_k\chi U_n\subset\R^2$ consists of a single horizontal line at height
	$p_kW_n(\frac{j_n^*}{n})$.
	The graph $\pi_E p_kW_n^{F'}\subset\R^2$ consists of horizontal line segments at height
	$p_kW_n(\frac{j}{n})$, $j=j_n^*,\dots,n-1$, together with interpolating
	vertical line segments.
	In particular,
	\[
	d_{E,[t_{n,N_n},1]}\big(\pi_E p_k W_n^{F'},\pi_E p_k\chi U_n \big)
	\le \max_{j_n^*\le j\le n} |p_k(W_n(\tfrac{j}{n})-W_n(\tfrac{j_n^*}{n}))|
	\]
	for each $k$ and so
	\[
	d_{E,[t_{n,N_n},1]}\big(\pi_E W_n^{F'},\pi_E\chi U_n \big)
	\le \max_{j_n^*\le j\le n} |W_n(\tfrac{j}{n})-W_n(\tfrac{j_n^*}{n})|.
	\]

	Also, 
	$(\pi_\tD\chi U_n)(t)=W_n(\frac{j_n^*}{n})$ for $t\in[j_n^*/n,1]$
	while the values of $\pi_\tD W_n^{F'}$ lie on the graph of $\pi_E W_n^{F'}$.
	It follows that
	\begin{align*}
	d_{\tD,[t_{n,N_n},1]}\big(\pi_\tD W_n^{F'},\pi_\tD\chi U_n \big) & \le
	\sup_{t_1,t_2\in [j_n^*/n,1]}|(\pi_\tD W_n^{F'})(t_1)-(\pi_\tD\chi U_n )(t_2)|
	\\ & \le \max_{j_n^*\le j\le n} |W_n(\tfrac{j}{n})-W_n(\tfrac{j_n^*}{n})|.
	\end{align*}
	Hence, it suffices to show that
	\[
	 \max_{j_n^*\le j\le n} |W_n(\tfrac{j}{n})-W_n(\tfrac{j_n^*}{n})|\to_{\mu_X} 0.
	\]

	To conclude, we use an argument from~\cite[Appendix~A]{Gouezel07}. 
	Passing to the natural extension, we may suppose without loss of generality that $T:M\to M$ is invertible. 
	Define measurable functions $m:M\to\N$, $v^*:M\to \R^d$ by
	\[
	m(x)=\inf\{k\ge 0:T^{-k}x\in X\}, \qquad 
	v^*(x)=\sum_{\ell=0}^{R(T^{-m}x)}|v(T^\ell(T^{-m}x))|.
	\]
	Notice that at time $t=1$ the process $W_n$ is in the middle of an excursion involving the increment $v\circ T^n$, while $v^*\circ T^n$ is the sum of the absolute values of the increments in that excursion. It follows that
	\[
	|W_n(\tfrac{j}{n})-W_n(\tfrac{j_n^*}{n})| \le b_n^{-1}v^*\circ T^n
	\]
	for $j_n^*\le j\le n$.
	Hence
	\[
	 \max_{j_n^*\le j\le n} |W_n(\tfrac{j}{n})-W_n(\tfrac{j_n^*}{n})|\le
	  b_n^{-1}v^*\circ T^n.
	\]

	Now, $b_n^{-1}v^*\to0$ a.e.\ on $(M,\mu)$ and $v^*\circ T^n=_\mu v^*$, so
	$b_n^{-1}v^*\circ T^n\to_\mu 0$.
	Also, $\mu_X=(\mu(X))^{-1}\mu|_X$ and hence
	$b_n^{-1}v^*\circ T^n\to_{\mu_X} 0$.
	 It follows that $\max_{j_n^*\le j\le n} |W_n(\frac{j}{n})-W_n(\frac{j_n^*}{n})|\to_{\mu_X} 0$
	as required.
	\end{proof}

	\subsection{Completed excursions on $[0,t_{n,N_n}]$}
	\label{sec:comp}

	Recall that $\chi U_n \in F'$ is defined by adjoining excursions that are scaled versions of the profiles $P_i$, $i\in\cI$. It is convenient also to define elements
	$\tU_n\in F'$ by adjoining dynamical excursions. Accordingly, define
	\[
	\tU_n=\big(U_n,\tcS_n,\{\te_{U_n}^\tau\}_{\tau\in\cS_{U_n}}\big)
	\]
	where $\tcS_n=\{t_{n,1},\dots,t_{n,N_n}\}$ and
	\[
	(\te_{U_n}^\tau)(t)=U_n(t_{n,j})+b_n^{-1}v_{[tR]}\circ f^j
	\quad\text{for $\tau=t_{n,j+1}$.}
	\]

	In the next two propositions, we consider the distances
	$d_{F'}(W_n^{F'},\tU_n)$,
	and
	$d_{F'}(\tU_n,\chi U_n)$
	on the interval $[0,t_{n, N_n}]$.

	\begin{prop}  \label{prop:comp1}
	$d_{F',[0,t_{n,N_n}]}( W_n^{F'},\tU_n)\to_{\mu_X}0$.
	\end{prop}

	\begin{proof}
	First, we consider 
	$d_{E,[0,t_{n,N_n}]}(\pi_E W_n^{F'},\pi_E\tU_n)$.

	Let $k\in\{1,\dots,d\}$.
	The graph $\pi_E p_k W_n^{F'}$ consists of the graph of $p_k W_n$ together with vertical line segments joining
	$(\frac{j}{n}, b_n^{-1}p_k v_{j-1})$ to $(\frac{j}{n}, b_n^{-1}p_k v_j)$ for $j=1,\dots,{n-1}$.
	Define
	\[
	q_{\rm min}=\min_{0\le\ell\le R}p_k v_\ell,
	\quad q_{\rm max}=\max_{0\le\ell\le R} p_k v_\ell.
	\]
	Then
	 the graph $\pi_E p_k \tU_n$ is obtained 
	from the graph of $p_k U_n$ by adjoining the line segments\footnote{We use the abbreviation $x+c[u_1,u_2]\circ f^j$ for $[x+c(u_1\circ f^j),x+c(u_2\circ f^j)]$.}
	\[
	\{t_{n,j+1}\}\times J_{n,j},
	\qquad J_{n,j}=
	p_kU_n(t_{n,j})+b_n^{-1}[q_{\rm min},
	q_{\rm max}]\circ f^j
	\]
	for $j=0,\dots,N_n-1$.
	The graphs are shown schematically in Figure~\ref{fig:comp1}.

	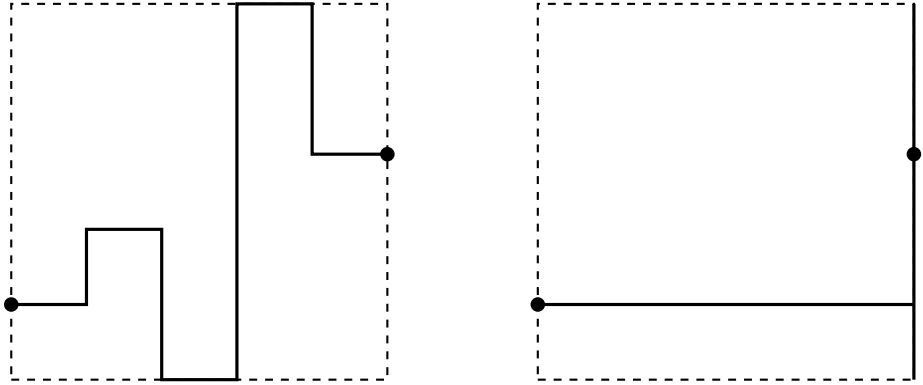
\begin{figure}[h]
	\centering
	\begin{tikzpicture}[thick, scale=0.5]


	\draw[ dashed] (0,0) -- (10,0)--(10, 10)--(0, 10)--(0,0);
	\draw[-, very thick] (0, 2)--(2, 2)--(2, 4)--(4,4)--(4, 0)--(6, 0)--(6, 10)--(8, 10)--(8, 6)--(10, 6);
	\fill (0,2) circle (5.5pt);
	\fill (10,6) circle (5.5pt);


	\draw[ dashed] (14,0) -- (24,0)--(24, 10)--(14, 10)--(14,0);
	\draw[-, very thick] (14, 2)--(24, 2);
	\draw[-, very thick] (24, 0)--(24, 10);
	\fill (14,2) circle (5.5pt);
	\fill (24,6) circle (5.5pt);

	\end{tikzpicture}

	\caption{Graphs of $\pi_E p_k W_n^{F'}$ (left) and $\pi_E p_k \tU_n$ (right) on an interval $(t_{n,j},t_{n,j+1}]$ of length $\frac{5}{n}$.
	The dashed lines show the box $[t_{n,j},t_{n,j+1}]\times J_{n,j}$.
	The dots show the points $(t_{n,j},W_n(t_{n,j}))$ and
	$(t_{n,j+1},W_n(t_{n,j+1}))$.}
	\label{fig:comp1}
	\end{figure}

	On the interval $(t_{n,j},t_{n,j+1}]$, $0\le j\le N_n-1$,
	the graphs $\pi_E p_kW_n^{F'}$ and $\pi_E p_k\tU_n$ lie entirely within the box 
	\[
	[t_{n,j},t_{n,j+1}]\times J_{n,j}.
	\]
	This box has width $n^{-1}R\circ f^j$.
	The graph $\pi_E p_k \tU_n$ contains $J_{n,j}$ which is the right-hand side of this box.
	Hence every point in the graph $\pi_E p_k W_n^{F'}$ lies within distance 
	$n^{-1}R\circ f^j$ of $\pi_E p_k \tU_n$.
	On the other hand, $J_{n,j}$ is by definition the union of horizontal translates of vertical line segments in $\pi_E p_k W_n^{F'}$ so every point in $J_{n,j}$ lies within distance $n^{-1}R\circ f^j$ of $\pi_E p_k W_n^{F'}$. The remaining horizontal line segment 
	$(t_{n,j},t_{n,j+1}]\times \{p_kU_n(t_{n,j})\}$ 
	in $\pi_E p_k \tU_n$
	is within distance $n^{-1}R\circ f^j$ of the point 
	$(t_{n,j},p_kU_n(t_{n,j}))$  which also lies on $\pi_E p_k W_n^{F'}$.
	Altogether, we have shown that
	$d_{E,(t_{n,j},t_{n,j+1}]}(\pi_E p_k W_n^{F'},\pi_E p_k \tU_n)\le n^{-1}R\circ f^j$.
	Hence
	\[
	d_{E,[0,t_{n,N_n}]}(\pi_E W_n^{F'},\pi_E\tU_n)\le 
	n^{-1}\max_{0\le j\le N_n-1}R\circ f^j\le  n^{-1}\max_{0\le j\le n}R\circ f^j.
	\]
	By Proposition~\ref{prop:ET},
	 $n^{-1}\max_{0\le j\le n}R\circ f^j\to0$  a.e.\ on $(X,\mu_X)$,
	and it follows that 
	\[
	d_{E,[0,t_{n,N_n}]}(\pi_E W_n^{F'},\pi_E\tU_n)\to_{\mu_X} 0.
	\]

	It remains to consider $d_{\tD,[0,t_{n,N_n}]}(\pi_\tD W_n^{F'},\pi_\tD\tU_n)$.
	Restricting to $(t_{n,j},t_{n,j+1}]$, $0\le j\le  N_n-1$,
	the function $\pi_\tD \tU_n$ is a concatenation of the constant function $U_n(t_{n,j})$
	followed by the excursion $U_n(t_{n,j})+b_n^{-1}v_{[tR]}\circ f^j$. The latter is a concatenation of the functions $U_n(t_{n,j})+b_n^{-1}v_\ell\circ f^j$ for $0\le\ell\le R$.
Hence $\pi_\tD \tU_n$ is a concatenation of the functions $U_n(t_{n,j})+b_n^{-1}v_\ell\circ f^j$ for $0\le\ell\le R$.
But $\pi_\tD W_n^{F'}$ is a concatenation of the same functions and in the same order. Since concatenation is associative under reparametrisations of time,
$\pi_{\tD,(t_{n,j},t_{n,j+1]}]}(\pi_\tD W_n^{F'},\pi_\tD\tU_n)=0$.
Hence
$d_{\tD,[0,t_{n,N_n}]}(\pi_\tD W_n^{F'},\pi_\tD\tU_n)=0$ completing the proof.
\end{proof}

\begin{prop}  \label{prop:comp2}
$d_{F',[0,t_{n,N_n}]}(\chi U_n, \tU_n)\to_{\mu_X}0$.
\end{prop}

\begin{proof}
We claim that
\[
d_{F',[0,t_{n,N_n}]}(\chi U_n,\tU_n)
\le b_n^{-1}\max_{0\le j\le n}d_\tD(\xi,\zeta)\circ f^j.
\]
The result then follows by hypothesis~\eqref{eq:main}.

It remains to prove the claim.
In general, $\disc_{U_n}\subset\tcS_n=\{t_{n,j}:1\le j\le N_n\}$.
We first consider the slightly simpler case 
$\disc_{U_n}=\tcS_n$ for all $n$. Then
\[
\chi U_n=(U_n,\tcS_n,\{e_{U_n}^\tau\}_{\tau\in \cS_{U_n}}), \qquad
\tU_n=(U_n,\tcS_n,\{\te_{U_n}^\tau\}_{\tau\in \cS_{U_n}}),
\]
where 
\begin{align*}
e_{U_n}^\tau(t) & = U_n(t_{n,j})+ |\Delta U_n(\tau)| \Pi(\Delta U_n(\tau))(t)  + t\big\{\Delta U_n(\tau) -|\Delta U_n(\tau)|\Pi(\Delta U_n(\tau))(1) \big\},
\\
\te_{U_n}^\tau(t) & =U_n(t_{n,j})+b_n^{-1}v_{[tR]}\circ f^j,
\end{align*}
for $\tau=t_{n,j+1}$.

Note that $\Delta U_n(\tau)= b_n^{-1}V\circ f^j$.
Hence
\begin{align*}
e_{U_n}^\tau(t) & 
=U_n(t_{n,j})+b_n^{-1}\big\{
|V| \Pi(V)(t)  + t\big\{V-|V|\Pi(V)(1)
\big\} \circ f^j
\\ & =U_n(t_{n,j})+b_n^{-1} \zeta(t)\circ f^j.
\end{align*}
Also,
\[
\te_{U_n}^\tau(t)  =U_n(t_{n,j})+b_n^{-1}\xi(t)\circ f^j.
\]

On the interval $(t_{n,j},t_{n,j+1}]$, it follows that
$\pi_\tD \tU_n$ is the concatenation of the constant function $U_n(t_{n,j})$ with $U_n(t_{n,j})+b_n^{-1} \xi(t) \circ f^j$, while
$\pi_\tD \chi U_n$ is the concatenation of the constant function $U_n(t_{n,j})$ with $U_n(t_{n,j})+b_n^{-1}\zeta(t)\circ f^j$. 
Hence 
\[
d_{\tD,(t_{n,j},t_{n,j+1}]}(\pi_\tD \tU_n,\pi_\tD \chi U_n)
\le b_n^{-1}d_\tD(\xi,\zeta)\circ f^j
\]
and it follows that
\[
d_{\tD,[0,t_{n,N_n}]}(\pi_\tD \tU_n,\pi_\tD \chi U_n)
\le b_n^{-1}\max_{0\le j\le n}d_\tD(\xi,\zeta)\circ f^j.
\]

Also,
$d_{E,(t_{n,j},t_{n,j+1}]}(\pi_E \tU_n,\pi_E \chi U_n)=b_n^{-1}H\circ f^j$
where $H$ is the Hausdorff distance between the smallest closed boxes containing 
the ranges of the functions $\xi(t)$
and $\zeta(t)$ for $t\in[0,1]$.
(So $H=\max_{1\le k\le d}H_k$ where $H_k$ is the Hausdorff distance between 
the smallest closed interval containing 
$\{p_k\xi(t):t\in[0,1]\}$ and
the smallest closed interval containing 
$\{p_k\zeta(t):t\in[0,1]\}$.)
In particular,
\[
H\le\sum_{k=1}^d\Big\{\Big|\max_{[0,1]}p_k\xi-\max_{[0,1]}p_k\zeta\Big|\vee
\Big|\min_{[0,1]}p_k\xi-\min_{[0,1]}p_k\zeta\Big|\Big\}
\le d_\tD(\xi,\zeta).
\]
Again,
\[
d_{E,[0,t_{n,N_n}]}(\pi_E \tU_n,\pi_E \chi U_n)
= b_n^{-1}\max_{0\le j\le n}H\circ f^j
\le  b_n^{-1}\max_{0\le j\le n}d_\tD(\xi,\zeta)\circ f^j,
\]
completing the proof of the claim in the case $\disc_{U_n}=\tcS_n$ for all $n$. 

In general, there is the possibility that $t_{n,j+1}\in\tcS_n\setminus \disc_{U_n}$.
On the interval $(t_{n,j},t_{n,j+1}]$, it remains the case that
$\pi_\tD \tU_n$ is the concatenation of the constant function $U_n(t_{n,j})$ with $U_n(t_{n,j})+b_n^{-1} \xi(t) \circ f^j$, while
$\pi_\tD \chi U_n$ is simply the 
constant function $U_n(t_{n,j})$. The latter is equivalent to
the concatenation of $U_n(t_{n,j})$ with 
$U_n(t_{n,j})$. 
But the fact that $t_{n,j+1}$ is not a discontinuity point of $U_n$ means that $V=0$ and hence (by definition) $\zeta(t)\circ f^j=0$.
Hence
$\pi_\tD \chi U_n$ is still the
concatenation of 
$U_n(t_{n,j})$ with
$U_n(t_{n,j})+b_n^{-1}\zeta(t)\circ f^j$. 
\end{proof}

\begin{cor} \label{cor:comp} 
$d_{F',[0,t_{n,N_n}]}(W_n^{F'}, \chi U_n)\to_{\mu_X}0$.
\end{cor}

\begin{proof} 
This is immediate from Propositions~\ref{prop:comp1} and~\ref{prop:comp2}.
\end{proof}

Thus we have completed the proof of Theorem~\ref{thm:main}.

\section{Examples}
\label{sec:ex}

In this section, we consider examples covered by this paper, expanding on the examples discussed in Section~\ref{sec:informal}.
In Subsection~\ref{sec:Gouezel}, we consider examples with $\alpha\in(0,1)$ where $T:[0,1]\to[0,1]$ is a uniformly expanding map and $v$ is an unbounded scalar observable.
In Subsection~\ref{sec:Young}, we consider examples with $\alpha\in(1,2)$ where $T:M\to M$ is only nonuniformly expanding/hyperbolic but $v:M\to\R^d$ is bounded.

\subsection{Examples with unbounded observables}
\label{sec:Gouezel}

In this subsection, we give details for Examples~\ref{ex:3} and~\ref{ex:5}.

\begin{example}[Example~\ref{ex:3} revisited] \label{ex:3+}
Let $T:M\to M$ be the doubling map, so $M=[0,1]$ and 
$Tx=2x\bmod1$, with ergodic probability measure $\mu=\Leb$.
Fix $\alpha\in(0,1)$ and consider the observable 
\[
v:M\to\R, \qquad v(x)=x^{-1/\alpha}.
\]
Define $c=2^{-1/\alpha}\in(0,1)$.
Let $\cI=\{1\}$ and let $P_1:[0,1]\to\R$ be any monotone increasing step function
with range precisely $\{1-c^j,\,j=0,1,2,\dots\}\cup\{1\}$.
In particular, $P_1(0)=0$, $P_1(1)=\omega_1=1$.

Define $\tL_\alpha$ to be the totally-skewed $\alpha$-stable L\'evy process with
spectral measure $\nu=\delta_1$ as defined in Section~\ref{sec:rv}, and 
let $L_\alpha=c\tL_\alpha$.
Take $b_n=(1-c)^{-1}n^{1/\alpha}$ and define $W_n\in D$ as in~\eqref{eq:Wn}.

By~\cite{Gouezel_doub}, $W_n(1)\to_\mu L_\alpha(1)$.
We claim that $W_n\to_\mu L_\alpha$ in $\cM_1$. Moreover, we have the following convergence result in $F'$:
Define the enriched process
$L_\alpha^{F'}=(L_\alpha,\disc_{L_\alpha},\{e^\tau\})\in F'$ with excursions
\[
e^\tau(t)=L_\alpha(\tau^-)+\Delta L_\alpha(\tau)P_1(t)
\]
at each discontinuity $\tau\in \disc_{L_\alpha}$.
Also, define $W_n^{F'}\in F'$ by attaching trivial excursions as in Section~\ref{sec:main}.
We prove that 
\[
W_n^{F'}\to_\mu L_\alpha^{F'} \quad\text{in $F'$}
\]
by verifying the assumptions of Theorem~\ref{thm:main}.
This is done in Lemmas~\ref{lem:ex3a} and~\ref{lem:ex3b} below.

\vspace{1ex}
The first return map $f=T^R:X\to X$, $X=[\frac12,1]$, is uniformly expanding with (countably many) full branches of constant slope,
and $\mu_X$ is normalised Lebesgue measure on $X$.

For $x\in X$ and $1\le \ell\le R(x)$,
\[
v(T^\ell x)=(2^{\ell -1}Tx)^{-1/\alpha}=
(2^\ell(x-\tfrac12))^{-1/\alpha}=
(x-\tfrac12)^{-1/\alpha}c^{\ell},
\]
so 
\[
v_\ell(x)  =v(x)+(c^{-1}-1)^{-1}(1-c^{\ell-1})(x-\tfrac12)^{-1/\alpha}.
\]
Also, $x-\tfrac12\in(2^{-(R(x)+1)},2^{-R(x)}]$, so
\[
V(x)=v_R(x)= (c^{-1}-1)^{-1}(x-\tfrac12)^{-1/\alpha} + O(1).
\]

\begin{lemma}  \label{lem:ex3a}
$W_n^V\to_{\mu_X} \tL_\alpha$ in the $\cJ_1$ topology.
\end{lemma}

\begin{proof}
Write $V=Z+H$ where $Z(x)=(c^{-1}-1)^{-1}(x-\tfrac12)^{-1/\alpha}$ and $H=O(1)$.
Since $H$ is bounded and H\"older, it follows by~\cite{MTorok12} or~\cite[Proposition~7.1]{KellyM16} that
$\big|\max_{k\le n}|\sum_{0\le j\le k}H\circ f^j|\big|_1\ll n^{1/2}$.
Hence it suffices to show that
$W_n^{Z}\to_{\mu_X} \tL_\alpha$ in the $\cJ_1$ topology.

Now, $Z(x)>t$ for $t>0$ large if and only if
$\frac12 < x < \frac12+((c^{-1}-1)t)^{-\alpha}$, so
$\mu_X(Z>t)=2\Leb(Z>t)=2 ((c^{-1}-1)t)^{-\alpha}$.
Hence $Z$ is regularly varying of order $\alpha$
and lies in the domain of attraction of a stable law $G_\alpha$ with spectral measure $\nu=\delta_1$, and $b_n=2^{1/\alpha}(c^{-1}-1)^{-1}n^{1/\alpha}
=(1-c)^{-1/\alpha}n^{1/\alpha}$ as described in Section~\ref{sec:rv}.

To obtain convergence of $W_n^Z$ to the corresponding L\'evy process $\tL_\alpha$, we apply~\cite[Theorem~1.2]{TyranKaminska10}.
On $(0,\infty)\times(\R\setminus\{0\})$,
define the sequence of random point processes
$\cN_n=\sum_{j=1}^n \delta_{(\frac{j}{n},b_n^{-1} Z\circ f^{j-1})}$
and the Poisson point process $\cN$ with mean measure $\Leb\times\Pi$ where
$\Pi(B)=\alpha\int_0^\infty 1_B(r)r^{-\alpha-1}dr$.
Since $\alpha\in(0,1)$, it suffices by~\cite[Theorem~1.2]{TyranKaminska10} to
show that $\cN_n\to_{\mu_X}\cN$.
 This holds by~\cite[Theorem~4.3]{Freitas2Magalhaes20}.
\end{proof}

\begin{lemma}  \label{lem:ex3b}
Hypothesis~\eqref{eq:main} is satisfied.
\end{lemma}

\begin{proof}
Let $t\in[0,1]$, $x\in X$. Since $V=v_R>0$, it follows from the calculations above and Remark~\ref{rmk:zeta} that
\begin{align*}
\xi(t)(x) & =v_{[tR]}(x)= (x-\tfrac12)^{-1/\alpha}(c^{-1}-1)^{-1}g_x(t)+O(1), \\
\zeta(t)(x) & =v_R(x)P_1(t)= (x-\tfrac12)^{-1/\alpha}(c^{-1}-1)^{-1}P_1(t)+O(1),
\end{align*}
where $g_x(t)=1-c^{[tR(x)]-1}$.

For all $x\in X$, the
functions $g_x$, $P_1$ are monotone increasing on $[0,1]$.
Moreover, there are intervals $[0,t_0(x)]$, $[0,t_1(x)]$ such that
${g_x}|_{[0,t_0(x)]}$ and ${P_1}|_{[0,t_1(x)]}$ take precisely the same values, namely
$\{0,1-c,1-c^2,\dots,1-c^{R(x)-1}\}$.
Furthermore, ${g_x}|_{[t_0(x),1]}\equiv 1-c^{R(x)-1}$
and $P_1([t_0(x),1)\subset [1-c^{R(x)-1},1]$.
Hence $d_\tD(g_x,P_1)\le c^{R(x)-1}$.
Using again that $(x-\tfrac12)^{-1/\alpha}\ll c^{-R(x)}$, it follows that
\[
\sup_{x\in X}d_\tD\big(\xi(\cdot)(x),\zeta(\cdot)(x)\big)<\infty.
\]
Hence hypothesis~\eqref{eq:main} is satisfied.
\end{proof}

\end{example}

\begin{example}[Example~\ref{ex:5} revisited] \label{ex:5+}
Let $T:M\to M$ be the tripling map, so $M=[0,1]$ and 
$Tx=3x\bmod1$, with ergodic probability measure $\mu=\Leb$.
Fix $\alpha\in(0,1)$ and consider the observable 
\[
v:M\to\R, \qquad v(x)=|x-\tfrac18|^{-1/\alpha}-|x-\tfrac38|^{-1/\alpha}.
\]
Define $c=3^{-1/\alpha}\in(0,1)$.
Set $\cI=\{\pm1\}$ and let $P_1:[0,1]\to\R$ be any step function
with values $\{1-(-c)^j,\,j=0,1,2,\dots\}\cup\{1\}$ taken in that order.
Let $P_{-1}=-P_1$. We have $\omega_\pm=\pm1$.

Define $\tL_\alpha$ to be the symmetric $\alpha$-stable L\'evy process with
spectral measure $\nu=\frac{1}{2}(\delta_1+\delta_{-1})$, and 
let $L_\alpha=(\tfrac79)^{1/\alpha}\tL_\alpha$.
Take $b_n=(\frac{36}{7})^{1/\alpha}(c^{-1}-1)^{-1}n^{1/\alpha}$ and define $W_n\in D$ as in~\eqref{eq:Wn}.
Define $W_n^{F'}\in F'$ by adjoining trivial profiles.

Define the enriched process
$L_\alpha^{F'}=(L_\alpha,\disc_{L_\alpha},\{e^\tau\})\in F'$ with excursions
\[
e^\tau(t)=L_\alpha(\tau^-)+\Delta L_\alpha(\tau)P_1(t)
\]
at each discontinuity $\tau\in \disc_{L_\alpha}$.
(Equivalently, attach suitably scaled profiles $P_1$ at positive jumps and $P_{-1}$ at negative jumps.)
We prove that 
\[
W_n^{F'}\to_\mu L_\alpha^{F'} \quad\text{in $F'$}
\]
by verifying the assumptions of Theorem~\ref{thm:main}, thereby recovering by a different method a result of~\cite[Example 2.7]{Freitas2Toddsub}.
This is done in Lemmas~\ref{lem:ex5a} and~\ref{lem:ex5b} below.

\vspace{1ex}
It is convenient to use
cylinder notation with letters $0,1,2$ denoting
$[0,\frac13]$,
$[\frac13,\frac23]$,
$[\frac23,1]$,
respectively.
So for example,
$[020]$ denotes the $3$-cylinder
$[0,\frac13]\cap T^{-1}[\frac23,1]\cap T^{-2}[0,\frac13]$.
We induce on the set 
$X=M\setminus([\frac19,\frac29]\cup[\frac13,\frac49])=
M\setminus([01]\cup[10])$.
Then long returns correspond to elements of
$[0(01)^n]\cup[1(10)^n]\cup[2(01)^n]\cup[2(10)^n]$ for $n$ large.
(Such points are, after one iterate, close to the periodic orbit $\{\frac18,\tfrac38\}$.)
As in Example~\ref{ex:3+}, the first return map $f=T^R:X\to X$ is uniformly expanding with full branches of constant slope,
and $\mu_X$ is normalised Lebesgue measure on $X$.

Write $X=X_1\dot\cup X_2$ where
$X_1=[00]\cup[11]\cup[20]\cup[21]$.
Since $R|_{X_2}=1$, our calculations focus on $x\in X_1$.

\begin{prop} \label{prop:vell}
Let $x\in X_1$, $1\le \ell\le R(x)$.
\begin{itemize}
\item[(i)] If $x\in[00]$, then
$v_\ell(x)=(c^{-1}+1)^{-1}(1-(-c)^{\ell-1})(x-\frac{1}{24})^{-1/\alpha}+O(\ell)$.
\item[(ii)] If $x\in[11]$, then
$v_\ell(x)=-(c^{-1}+1)^{-1}(1-(-c)^{\ell-1})(x-\frac{11}{24})^{-1/\alpha}+O(\ell)$.
\item[(iii)] If $x\in[20]$, then
$v_\ell(x)=(c^{-1}+1)^{-1}(1-(-c)^{\ell-1})(x-\frac{17}{24})^{-1/\alpha}+O(\ell)$.
\item[(iv)] If $x\in[21]$, then
$v_\ell(x)=-(c^{-1}+1)^{-1}(1-(-c)^{\ell-1})(x-\frac{19}{24})^{-1/\alpha}+O(\ell)$.
\end{itemize}
\end{prop}

\begin{proof}
Define $b_\ell(x)=\begin{cases} \tfrac38 & T^{\ell-1} x\in[0] \\
 \tfrac18 & T^{\ell-1}x \in[1] 
\end{cases}\;.$
Inductively, for $x$ in $\ell$-cylinders $[0101\cdots]$, $[1010\cdots]$,
we have $T^\ell x=3^\ell(x-a(x))+b_\ell(x)$
where $a(x)$ takes values $\frac18,\frac38$ depending on whether $x\in[0]$ or $x\in[1]$.
It follows that
\[
T^\ell x=3^\ell (x-\tfrac{1}{24})+b_\ell(x)
\]
 for $x$ in
$\ell$-cylinders $[0 0101\ldots]$,

Let $x\in[00]$. For $1\le k<R(x)$,
\begin{align*}
v(T^k x) & = \begin{cases}
(T^k x-\tfrac18)^{-1/\alpha}+O(1) & k \text{ odd} \\
-(T^k x-\tfrac38)^{-1/\alpha}+O(1) & k \text{ even}
\end{cases}
\\ & = -(1)^{k+1}(3^k(x-\tfrac{1}{24}))^{-1/\alpha}+O(1)
=-(-c)^k (x-\tfrac{1}{24})^{-1/\alpha}+O(1).
\end{align*}
Hence $v_\ell(x)=-\sum_{k=1}^{\ell-1}(-c)^k(x-\tfrac{1}{24})^{-1/\alpha}+O(\ell)$ completing the proof of (i). The other cases are similar.
\end{proof}

Similarly to Example~\ref{ex:3+}, it follows that
\[
V(x)=v_R(x)= \pm (c^{-1}+1)^{-1}(x-a)^{-1/\alpha}1_{X_1}(x) + O(R(x)),
\]
for the appropriate choices of $\pm$ and $a\in\{\frac{1}{24},\frac{11}{24},\frac{17}{24},\frac{19}{24}\}$.

\begin{lemma}  \label{lem:ex5a}
$W_n^V\to_{\mu_X} \tL_\alpha$ in the $\cJ_1$ topology.
\end{lemma}

\begin{proof}
Write $V=Z+H$ where $Z(x)=\pm (c^{-1}+1)^{-1}(x-a)^{-1/\alpha}1_{X_1}(x)$ and $H=O(R)$.
Since $H$ is H\"older and in $L^p$ for all $p<\infty$, it follows by~\cite{MTorok12} or~\cite[Proposition~7.1]{KellyM16} that
$\big|\max_{k\le n}|\sum_{0\le j\le k}H\circ f^j|\big|_1\ll n^{1/2}$.
Hence it suffices to show that
$W_n^{Z}\to_{\mu_X} \tL_\alpha$ in the $\cJ_1$ topology.

Suppose that $x\in[00]$.
Then $Z(x)>t$ for $t>0$ large if and only if
$\frac{1}{24} < x < \frac{1}{24}+((c^{-1}+1)t)^{-\alpha}$.
A similar estimate holds for $x\in X_1\setminus[00]$, so
$\mu_X(Z>t)=\mu_X(Z<-t)=\frac{9}{7}\Leb(Z<-t)= \frac{18}{7}((c^{-1}+1)t)^{-\alpha}$.
Hence $Z$ is regularly varying of order $\alpha$
and lies in the domain of attraction of a stable law $G_\alpha$ with spectral measure $\nu=\frac12(\delta_1+\delta_{-1})$, and $b_n=(\frac{36}{7})^{1/\alpha}(c^{-1}-1)^{-1}n^{1/\alpha}$ as described in Section~\ref{sec:rv}.
Since $\alpha\in(0,1)$, the result follows (as in the proof of Lemma~\ref{lem:ex3a}) from~\cite[Theorem~1.2]{TyranKaminska10} and~\cite[Theorem~4.3]{Freitas2Magalhaes20}.
\end{proof}

\begin{lemma}  \label{lem:ex5b}
Hypothesis~\eqref{eq:main} is satisfied.
\end{lemma}

\begin{proof}
For $t\in[0,1]$, $x\in X_1$, it follows from the calculations above and Remark~\ref{rmk:zeta} that
\begin{align*}
\xi(t)(x) & =v_{[tR]}(x)= \pm (c^{-1}+1)^{-1}(x-a)^{-1/\alpha}g_x(t)+O(R), \\
\zeta(t)(x) & =|v_R(x)|P_{\sgn V}(t)= \pm (c^{-1}+1)^{-1}(x-a)^{-1/\alpha}P_1(t)+O(R),
\end{align*}
where $g_x(t)=(1-(-c)^{[tR(x)]-1})$,
with the appropriate (and matching) choices of $\pm$ and 
$a\in\{\frac{1}{24},\frac{11}{24},\frac{17}{24},\frac{19}{24}\}$.

For all $x\in X_1$, the
functions $g_x$, $P_1$ are piecewise constant on $[0,1]$.
Moreover, there are intervals $[0,t_0(x)]$, $[0,t_1(x)]$ such that
${g_x}|_{[0,t_0(x)]}$ and ${P_1}|_{[0,t_1(x)]}$ take precisely the same values, namely
$\{0,1+c,1-c^2,\dots,1-(-c)^{R(x)-1}\}$.
Furthermore, ${g_x}|_{[t_0(x),1]}\equiv 1-(-c)^{R(x)-1}$
and $P_1([t_0(x),1)\subset [1-(-c)^{R(x)-1},1]$.
Hence $d_\tD(g_x,P_1)\le c^{R(x)-1}$.
Using again that $(x-a)^{-1/\alpha}\ll c^{-R(x)}$, it follows that
\[
d_\tD\big(\xi(\cdot)(x),\zeta(\cdot)(x)\big)\ll R(x)
\]
on $X_1$ and hence on $X$.
Since $R\in L^p$ for all $p<\infty$,  hypothesis~\eqref{eq:main} follows from Proposition~\ref{prop:ET}.
\end{proof}

\end{example}

\subsection{Examples with bounded observables}
\label{sec:Young}

In this subsection, we consider examples where the underlying dynamical system 
$T:M\to M$ is nonuniformly/hyperbolic expanding with a better-behaved first return map $f=T^R:X\to X$, and the observable $v:M\to\R^d$ is bounded.

We continue to assume the setup at the beginning of Section~\ref{sec:main}
with integrable return time $R$.
Moreover, we assume that there is a finite disjoint collection $\{X_i,\,i\in\cI\}$ of subsets of $X$ ($\cI\neq\emptyset$) such that
$\mu_X(R1_{X_i}>t)=c_i\mathcal{l}(t)t^{-\alpha}$, $c_i>0$, for each $i\in\cI$,
where $\alpha\in(1,2)$ and $\mathcal{l}:(0,\infty)\to(0,\infty)$ is continuous and slowly varying.
(In particular, $R\in L^1$ and $R\not\in L^2$.) 

Let $P_i:[0,1]\to\R^d$ be a finite collection of H\"older continuous profiles with
$P_i(0)=0$ and
$\omega_i=P_i(1)\in\bbS^{d-1}$ distinct.
For notational convenience, suppose that $0\not\in\cI$.
Set $X_0=X\setminus\bigcup_{i\in\cI}X_i$.
(It is permitted that $X_0=\emptyset$.) 

Let $v:M\to\R^d$ be an $L^\infty$ observable with $\int_M v\,d\mu=0$.
We assume that there exists $\eta>0$ and 
nonnegative $H\in L^p(X)$ for some $p>\alpha$ such that
\begin{equation} \label{eq:vl}
v_\ell=\sum_{i\in\cI}\{\lambda_i P_i(\ell/R)R +O(R^{1-\eta})\} 1_{X_i} +O(H) 1_{X_0} ,\quad \ell=0,1,\dots,R,
\end{equation}
where $\lambda_i\in\R$, $\lambda_i\neq0$. 

In particular, 
\[
V=v_R =\sum_{i\in\cI}\{\lambda_i \omega_i R +O(R^{1-\eta})\} 1_{X_i} +O(H)1_{X_0}.
\]
Regular variation of $V$ reduces to regular variation of
$Z=\sum_{i\in\cI}\lambda_i \omega_i R 1_{X_i}$ and we deduce that
$V$ is regularly varying with order $\alpha$ and spectral measure
\begin{equation} \label{eq:nu}
\nu=\Big(\sum_{i\in\cI} c_i |\lambda_i|^\alpha\Big)^{-1}\sum_{i\in\cI} c_i |\lambda_i|^\alpha \delta_{\omega_i}.
\end{equation} 
Let $\tL_\alpha$ denote the L\'evy process with spectral measure $\nu$.

Note that $\mu_X(|V|>t)\sim \mu_X(|Z|>t)\sim \mathcal{l}(t)\sum_{i\in\cI}c_i\lambda_i^\alpha t^{-\alpha}$. 
Hence we
choose $b_n\sim (n\mathcal{l}(b_n)\sum c_i\lambda_i^\alpha)^{1/\alpha}$.

In many examples, as described below, it can be verified that~\eqref{eq:vl} holds and that $W_n^V\to_{\mu_X}\tL_\alpha$
(equivalently $W^Z_n\to_{\mu_X} \tL_\alpha$) in the $\cJ_1$ topology.
Hence, to apply Theorem~\ref{thm:main}, it remains to verify hypothesis~\eqref{eq:main}. This we do now.

\begin{prop} \label{prop:vl}
Hypothesis~\eqref{eq:main} is satisfied.
\end{prop}

\begin{proof}
We claim that there exists $\eta>0$ such that
$\sup_{[0,1]}|\xi-\zeta|\ll R^{1-\eta}+H$.
Choose $p>\alpha$ such that $R^{1-\eta}+H\in L^p$.
By Proposition~\ref{prop:ET},
\[
n^{-1/p}\max_{0\le j\le n}\sup_{[0,1]}|\xi-\zeta|\circ f^j\to0 \;a.e.
\]
But $b_n\gg n^{1/p}$, so $b_n^{-1}\max_{0\le j\le n}\sup_{[0,1]}|\xi-\zeta|\circ f^j\to0$ a.e.\ and the result follows.

It remains to prove the claim. 
On $X_0$, it is clear that $\xi(t)=O(H)$ and $\zeta(t)=O(H)$ for all $t$, 
so it suffices to work on $X'=\bigcup_{i\in\cI}X_i$.

Define $J_1,J_2:X'\to\cI$ by setting 
$J_1=i$ if $V/|V|$ is closest to some $\omega_i$ and $J_1=1$ otherwise. Let $J_2|_{X_i}=i$.
By definition of $J_1$ and $J_2$, there exists $c_0>0$ such that
\[
X'\cap\{J_1\neq J_2\}\subset \bigcup_{i\in\cI}\Big\{x\in X_i:\big|\tfrac{v_R(x)}{|v_R(x)|}-\omega_i\big|>c_0\Big\}.
\]
For $i\in\cI$, by~\eqref{eq:vl}, 
$1_{X_i}v_R=\lambda_i\omega_iR +O(R^{1-\eta})$,
so
\(
1_{X_i}\tfrac{v_R}{|v_R|}=\omega_i+O(R^{-\eta}).
\)
Hence there exists $c_1>0$ such that
\(
X'\cap \{J_1\neq J_2\} \subset \{R<c_1\}.
\)
In particular, $|v_R|1_{X'}1_{\{J_1\neq J_2\}}\le |v|_\infty c_1$. 
Altogether, $|V|1_{X'}1_{\{J_1\neq J_2\}}=O(R^{1-\eta})$.

In addition, by~\eqref{eq:vl}, 
\(
|V-|V|P_{J_2}(1)|= |V-|V|\omega_i| = O(R^{1-\eta})
\)
on $X_i$, and so
\begin{align*}
\zeta(t)  = |V| P_{J_1}(t)  +t\big\{ V-|V| P_{J_1}(1)\big\}
& = |V| P_{J_2}(t)  +t\big\{ V-|V| P_{J_2}(1)\big\}+O(R^{1-\eta})
\\ & = |V|  P_i(t)+  O(R^{1-\eta})
 =\lambda_i  P_i(t) R+O(R^{1-\eta}).
\end{align*}

Shrinking $\eta$ if necessary, we can suppose that each profile $P_i$ is $C^\eta$.
Given $t\in[0,1]$, write $t=\frac{\ell}{R}+s$ where $\ell\ge0$ is an integer and $s\in[0,\frac{1}{R})$.
On $X_i$,
\begin{align*}
\xi(t)  =v_{[tR]}=v_\ell
 &=\lambda_i P_i(t) R +(P_i(\ell/R)-P_i(t))R+O(R^{1-\eta})
\\ &=\lambda_i P_i(t) R+O(R^{1-\eta}).
\end{align*}
Hence
$\sup_{[0,1]}|\xi-\zeta|=O(R^{1-\eta})$ on $X'$ completing the proof of the claim.
\end{proof}

\begin{example}[Example~\ref{ex:4} revisited] \label{ex:4+}
We return to the example of an intermittent map
$T:M\to M$, $M=[0,1]$, with finitely many neutral fixed points
$x_1,\dots,x_k$ of neutrality $\alpha_1,\dots,\alpha_k$ where
$\min\alpha_j=\alpha_1=\alpha\in(1,2)$.

We induce on a set $X\subset M$ bounded away from the neutral fixed points so
that $f=T^R:X\to X$ is a full-branch Gibbs-Markov map (uniformly expanding with bounded distortion) and so that $X=X_1\cup\dots\cup X_k$ where
each $X_j$ is a union of partition elements for $f$ and
trajectories in $X_j$ pass close to $x_j$ before returning to $X$.
Then $R$ is regularly varying of order $\alpha$ and
$\mu(R1_{X_j}>t)\sim c_j t^{-\alpha_j}$ for some $c_j>0$. In particular,
$R1_{X_j}\in L^q$ for all $q<\alpha_j$.

Let $v:M\to\R^d$ be H\"older with mean zero such that $v(x_1)\neq0$.
A calculation as in~\cite{Gouezel04,MZ15,CFKM20} shows that on $X_j$,
\[
v_\ell = \ell v(x_j) +O(R^{1-\eta}), \quad 0\le \ell<R,
\]
for some $\eta>0$.
In particular, if $\alpha_j>\alpha$ or $v(x_j)=0$, then
$v_\ell 1_{X_j}\in L^p$ for some $p>\alpha$.

Define $\cI$ to be the set of indices $i\in\{1,\dots,k\}$ with $\alpha_i=\alpha$ and $v(x_i)\neq0$. For $i\in\cI$, write $v(x_i)=\lambda_i\omega_i$ where $\lambda_i>0$ and
$\omega_i\in\bbS^{d-1}$.
Combining the sets $X_i$ with common value of $\omega_i$, we can suppose without loss that the $\omega_i$, $i\in\cI$, are distinct.
Define $P_i(t)=t \omega_i$.
Then
\[
v_\ell  
 = \sum_{i\in\cI}\{ \lambda_i P_i(\ell/R)R +O(R^{1-\eta})\}1_{X_i} +O(H)1_{X_0}, \quad 0\le \ell<R,
\]
where $H\in L^p$ for some $p>\alpha$.

Hence, we have verified~\eqref{eq:vl}, so hypothesis~\eqref{eq:main} holds by
Proposition~\ref{prop:vl}.
Moreover,
\[
V=v_R=Z+H', \qquad 
Z=\sum_{i\in\cI}\lambda_i\omega_i R 1_{X_i},
\]
where $H'\in L^p$ for some $p>\alpha$.
Hence, $V$ is regularly varying with order $\alpha$ and spectral measure $\nu$ given by~\eqref{eq:nu}, and we take $b_n=(\sum_{i\in\cI}c_i\lambda_i^\alpha)^{1/\alpha}n^{1/\alpha}$.
Since $Z$ is regularly varying and piecewise constant, it follows as
in~\cite{TyranKaminska10,MZ15} for $d=1$ and~\cite{CFKM20} for $d\ge1$ that
$W_n^Z\to_{\mu_X} \tL_\alpha$ in the $\cJ_1$ topology.
Since $H'\in L^p$, it follows from Proposition~\ref{prop:ET} that $W_n^V\to_{\mu_X} \tL_\alpha$ in the $\cJ_1$ topology.
This completes the verification of the hypotheses of Theorem~\ref{thm:main}.
\end{example}

\begin{example}[Billiards with flat cusps revisited] \label{ex:cusps+}
Finally, we return to the example of billiards with flat cusps described in Section~\ref{sec:cusps}.
Following~\cite{JungPeneZhang20}, we induce on a set $X\subset M$ bounded away from the flat cusps and so that $X=X_1\cup\dots\cup X_k$ where trajectories in $X_j$ pass close to the $j$'th cusp before returning to $X$.
Given a H\"older mean zero observable $v:M\to\R^d$, we define the profiles $P_i$, $i\in\cI$, corresponding to flattest cusps as in Section~\ref{sec:cusps}.

For verification of~\eqref{eq:vl}, we refer to~\cite[Proposition~8.1]{MV20}.
(The calculation there is written in the case $d=1$, but extends immediately to $d\ge2$.)

Convergence of $W_n^V$ is more difficult than in the previous examples since the induced map $f=T^R:X\to X$ has unbounded distortion. Instead, it is nonuniformly hyperbolic with exponential tails in the sense of~\cite{Young98}.
Convergence of $W_n^V$ in the $\cJ_1$ topology is proved in~\cite{JungPeneZhang20} for $d=1$ and in~\cite[Section~5]{CKMprep} for $d\ge1$. 

Unlike in Example~\ref{ex:4+}, we generally require that the vectors $\omega_i=P_i(1)$ are distinct at distinct flattest cusps, since the profiles $P_i$ are typically different. (In Example~\ref{ex:4+}, each profile $P_i$ was determined by $\omega_i$.) Also, we can no longer disregard flattest cusps with $P_i(1)=0$ since $P_i$ could still be nontrivial. These issues require further attention and are the subject of work in progress.
\end{example}

 \paragraph{Acknowledgements}
 ACMF, JMF and MT were partially supported by FCT projects PTDC/MAT-PUR/4048/2021 and 2022.07167.PTDC, with national funds, and by CMUP, which is financed by national funds through FCT -- Funda\c{c}\~ao para a Ci\^encia e a Tecnologia, I.P., under the project with reference UIDB/00144/2020.

IM and MT are grateful to hospitality at the University of Porto
 where part of this work was done.


\begin{thebibliography}{99}

\bibitem[AT]{AvramTaqqu92}
F.~Avram and M.~S. Taqqu. Weak convergence of sums of moving averages in the
  {$\alpha$}-stable domain of attraction. \emph{Ann. Probab.} \textbf{20}
  (1992) 483--503.

\bibitem[BCD]{BalintChernovDolgopyat11}
P.~B{\'a}lint, N.~Chernov and D.~Dolgopyat. Limit theorems for dispersing
  billiards with cusps. \emph{Comm. Math. Phys.} \textbf{308} (2011) 479--510.

\bibitem[BL]{BL}
F. Bartumeus and S. A. Levin.
Fractal reorientation clocks: Linking animal behavior to statistical patterns of search.
\emph{Proc. Natl. Acad. Sci. USA} \textbf{105} (2008) 19072--19077.

\bibitem[BK]{BasrakKrizmanic14}
B.~Basrak and D.~Krizmani\'{c}. A limit theorem for moving averages in the
  {$\alpha$}-stable domain of attraction. \emph{Stochastic Process. Appl.}
  \textbf{124} (2014) 1070--1083.

\bibitem[BKS]{BasrakKrizmanicSegers12}
B.~Basrak, D.~Krizmani\'c and J.~Segers. A functional limit theorem for
  dependent sequences with infinite variance stable limits. \emph{Ann. Probab.}
  \textbf{40} (2012) 2008--2033.

\bibitem[BS]{BunimSinai81}
L.~A. Bunimovich and Y.~G. Sina{\u\i}. Statistical properties of {L}orentz gas
  with periodic configuration of scatterers. \emph{Comm. Math. Phys.}
  \textbf{78} (1980/81) 479--497.

\bibitem[BSC]{BunimovichSinaiChernov91}
L.~A. Bunimovich, Y.~G. Sina{\u\i} and N.~I. Chernov. Statistical properties
  of two-dimensional hyperbolic billiards. \emph{Uspekhi Mat. Nauk} \textbf{46}
  (1991) 43--92.

\bibitem[CM]{ChernovMarkarian}
N.~Chernov and R.~Markarian. \emph{Chaotic billiards}. Mathematical Surveys and
  Monographs \textbf{127}, American Mathematical Society, Providence, RI, 2006.

\bibitem[CFKM]{CFKM20}
I.~Chevyrev, P.~K. Friz, A.~Korepanov and I.~Melbourne. Superdiffusive limits
  for deterministic fast-slow dynamical systems. \emph{Probab. Theory Related
  Fields} \textbf{178} (2020) 735--770.

\bibitem[CKM]{CKMprep}
I.~Chevyrev, A.~Korepanov and I.~Melbourne.
Superdiffusive limits beyond the Marcus regime for deterministic fast-slow dynamical systems. 
Preprint
(\url{https://arxiv.org/abs/2312.15734}).

\bibitem[FFM]{Freitas2Magalhaes20}
A.~C.~M. Freitas, J.~M. Freitas and M.~Magalh\~{a}es. Complete convergence and
  records for dynamically generated stochastic processes. \emph{Trans. Amer.
  Math. Soc.} \textbf{373} (2020) 435--478.

\bibitem[FFT]{Freitas2Toddsub}
A.~C.~M. Freitas, J.~M. Freitas and M.~Todd. Enriched functional limit
  theorems for dynamical systems. Preprint
(\url{https://arxiv.org/abs/2011.10153}).

\bibitem[GW]{GaspardWang88}  P. Gaspard and X.-J. Wang.
Sporadicity: Between periodic and chaotic dynamical behaviors.
\emph{Proc. Natl. Acad. Sci. USA} \textbf{85} (1988) 4591--4595.

\bibitem[G1]{Gouezel_doub}
S.~Gou{\"e}zel. Stable laws for the doubling map. Unpublished notes.

\bibitem[G2]{Gouezel04}
S.~Gou{\"e}zel. {Central limit theorem and stable laws for intermittent maps}.
  \emph{Probab. Theory Relat. Fields} \textbf{128} (2004) 82--122.

\bibitem[G3]{Gouezel07}
S.~Gou{\"e}zel. {Statistical properties of a skew product with a curve of
  neutral points}. \emph{Ergodic Theory Dynam. Systems} \textbf{27} (2007)
  123--151.

\bibitem[J]{Jakubowski07}
A.~Jakubowski. The Skorokhod Space in functional convergence: a short
  introduction. \emph{International conference: Skorokhod Space. 50 years on,
  17-23 June 2007, Kyiv, Ukraine , Part I, s. 11-18.}, 2007.

\bibitem[JMP+]{JMPVZ21}
P.~Jung, I.~Melbourne, F.~P\`ene, P.~Varandas and H.-K. Zhang. Necessary and
  sufficient condition for $\cM_2$-convergence to a L\'evy process for
  billiards with cusps at flat points. \emph{Stoch. Dyn.} \textbf{21} (2021)
  2150024, 8 pages.

\bibitem[JPZ]{JungPeneZhang20}
P.~Jung, F.~P\`ene and H.-K. Zhang. Convergence to {$\alpha$}-stable
  {L}\'{e}vy motion for chaotic billiards with several cusps at flat points.
  \emph{Nonlinearity} \textbf{33} (2020) 807--839.

\bibitem[JZ]{JungZhang18}
P.~Jung and H.-K. Zhang. Stable laws for chaotic billiards with cusps at flat
  points. \emph{Annales Henri Poincar\'e} \textbf{19} (2018) 3815--3853.

\bibitem[KM]{KellyM16}
D.~Kelly and I.~Melbourne. {Smooth approximation of stochastic differential
  equations}. \emph{Ann. Probab.} \textbf{44} (2016) 479--520.

\bibitem[KGS+]{KGSSR} R. Klages, S. Gallegos, J. Solanp\"a\"a, M. Sarvilahti and
E R\"as\"anen.  Normal and anomalous diffusion in soft Lorentz gases.
\emph{Phys. Rev. Lett.} \textbf{122} (2019) 064102.

\bibitem[LSV]{LSV99}
C.~Liverani, B.~Saussol and S.~Vaienti. {A probabilistic approach to
  intermittency}. \emph{Ergodic Theory Dynam. Systems} \textbf{19} (1999)
  671--685.

\bibitem[L]{Lorentz05}
H~Lorentz. The motion of electrons in metallic bodies. \emph{Proc. Amst. Acad.}
  \textbf{7} (1905) 438--453.

\bibitem[MT]{MTorok12}
I.~Melbourne and A.~T{\" o}r{\" o}k. {Convergence of moments for Axiom A and
  nonuniformly hyperbolic flows}. \emph{Ergodic Theory Dynam. Systems}
  \textbf{32} (2012) 1091--1100.

\bibitem[MV]{MV20}
I.~Melbourne and P.~Varandas. Convergence to a L\'evy process in the Skorohod
  ${\mathcal M}_1$ and ${\mathcal M}_2$ topologies for nonuniformly hyperbolic
  systems, including billiards with cusps. \emph{Comm. Math. Phys.}
  \textbf{375} (2020) 653--678.

\bibitem[MZ]{MZ15}
I.~Melbourne and R.~Zweim{\"u}ller. {Weak convergence to stable L\'evy
  processes for nonuniformly hyperbolic dynamical systems}. \emph{Ann\ Inst. H.
  Poincar{\'e} (B) Probab. Statist.} \textbf{51} (2015) 545--556.

\bibitem[MJCB]{Metzler}
R. Metzler, J-H. Jeon, A. G. Cherstvy and E. Barkai.
 Anomalous diffusion models and their properties: non-stationarity, non-ergodicity, and ageing at the centenary of single particle tracking.
\emph{Phys. Chem. Chem. Phys.} \textbf{16} (2014) 24128--24164.

\bibitem[PVHS]{PVHS} B. Podobnik, A. Valentin\u{c}i\u{c}, D. Horvati\'c and
H. E. Stanley.
Asymmetric L\'evy flight in financial ratios.
\emph{Proc. Natl. Acad. Sci. USA} \textbf{108} (2011) 17883--17888.

\bibitem[PM]{PomeauManneville80}
Y.~Pomeau and P.~Manneville. Intermittent transition to turbulence in
  dissipative dynamical systems. \emph{Comm. Math. Phys.} \textbf{74} (1980)
  189--197.

\bibitem[ST]{SamorodnitskyTaqqu}
G.~Samorodnitsky and M.~S. Taqqu. \emph{Stable non-{G}aussian random
  processes}, Stochastic Modeling. Chapman \& Hall, New York, 1994. Stochastic
  models with infinite variance.

\bibitem[S]{Sinai70}
Y.~G. Sina{\u\i}. Dynamical systems with elastic reflections. {E}rgodic
  properties of dispersing billiards. \emph{Uspehi Mat. Nauk} \textbf{25}
  (1970) 141--192.

\bibitem[Sk]{Skorohod56}
A.~V. Skorohod. Limit theorems for stochastic processes. \emph{Teor.
  Veroyatnost. i Primenen.} \textbf{1} (1956) 289--319.

\bibitem[SWS]{Swinney}
T. H. Solomon, E. R. Weeks and H. L. Swinney.
Chaotic advection in a two-dimensional flow: L\'evy flights and anomalous diffusion.
\emph{Phys. D} \textbf{76} (1994) 70--84.

\bibitem[T]{TyranKaminska10}
M.~Tyran-Kami{\'n}ska. Weak convergence to {L}\'evy stable processes in
  dynamical systems. \emph{Stoch. Dyn.} \textbf{10} (2010) 263--289.

\bibitem[W]{Whitt}
W.~Whitt. \emph{Stochastic-process limits}. Springer Series in Operations
  Research, Springer-Verlag, New York, 2002. An introduction to
  stochastic-process limits and their application to queues.


\bibitem[Y]{Young98}
L.-S. Young. Statistical properties of dynamical systems with some
  hyperbolicity. \emph{Ann. of Math.} \textbf{147} (1998) 585--650.

\bibitem[Z]{Zweimuller07}
R.~Zweim{\"u}ller. Mixing limit theorems for ergodic transformations. \emph{J.
  Theoret. Probab.} \textbf{20} (2007) 1059--1071.


\end{thebibliography}
\end{document}